\documentclass[11pt,a4paper, reqno]{amsart}
\usepackage[margin=1.15in]{geometry}
\usepackage{amsmath, amssymb,amsmath,amscd,amsfonts,amsthm,mathrsfs, graphicx}
\usepackage[arrow,matrix,curve,cmtip,ps]{xy}
\usepackage{enumitem}
\usepackage[alphabetic]{amsrefs}
\usepackage{pb-diagram} 
\usepackage{pb-xy}
\usepackage{color}
\usepackage{xcolor}
\usepackage{hyperref}
\usepackage{tikz}
\usepackage{pinlabel}

\usepackage{pinlabel}	

\newtheorem{proposition}{Proposition}[section]
\newtheorem{theorem}[proposition]{Theorem}
\newtheorem{corollary}[proposition]{Corollary}

\newtheorem{thmx}{Theorem}

\newtheorem{thmxprime}{Theorem}

\newtheorem*{theorem*}{Theorem}
\newtheorem*{proposition*}{Proposition}
\newtheorem*{lemma*}{Lemma}
\newtheorem*{corollary*}{Corollary}
\theoremstyle{definition}
\newtheorem{definition}[proposition]{Definition}

\newtheorem{remark}[proposition]{Remark}

\newcommand{\Int}{\operatorname{Int}}
\newcommand{\Hom}{\operatorname{Hom}}
\newcommand{\rk}{\operatorname{rk}}
\newcommand{\tb}{\operatorname{tb}}
\newcommand{\rot}{\operatorname{rot}}

\newcommand{\Q}{\mathbb{Q}}
\newcommand{\N}{\mathbb{N}}
\newcommand{\Z}{\mathbb{Z}}
\newcommand{\F}{\mathcal{F}}
\newcommand{\Fn}{\mathcal{F}_n}
\newcommand{\Pn}{\mathcal{P}_n}
\newcommand{\Nn}{\mathcal{N}_n}
\newcommand{\Fnpointfive}{\mathcal{F}_{n.5}}

\newcommand{\B}{\mathcal{B}}
\newcommand{\Bn}{\mathcal{B}_n}

\newcommand{\C}{\mathcal{C}}
\newcommand{\K}{\mathcal{K}}
\newcommand{\A}{\mathcal{A}}

\newcommand{\Arf}{\operatorname{Arf}}

\newcommand{\Bl}{\mathcal{B}\ell}

\DeclareMathOperator{\lk}{lk}

\newcommand{\into}{\hookrightarrow}

\numberwithin{equation}{section}

\begin{document}
\title[Linear independence of cables in the knot concordance group]{Linear independence of cables in the knot concordance group}

\author{Christopher W.\ Davis}
\address{University of Wisconsin--Eau Claire}
\email{daviscw@uwec.edu}
\urladdr{people.uwec.edu/daviscw}

\author{JungHwan Park}
\address{Max-Planck-Institut f\"{u}r Mathematik}
\email{jp35@mpim-bonn.mpg.de }
\urladdr{http://people.mpim-bonn.mpg.de/jp35/}

\author{Arunima Ray}
\address{Max-Planck-Institut f\"{u}r Mathematik}
\email{aruray@mpim-bonn.mpg.de }
\urladdr{http://people.mpim-bonn.mpg.de/aruray/}

\date{\today}


\begin{abstract} 
We produce infinite families of knots $\{K^i\}_{i\ge 1}$ for which the set of cables $\{K^i_{p,1}\}_{i,p\ge 1}$ is linearly independent in the knot concordance group, $\C$.  We arrange that these examples lie arbitrarily deep in the solvable and bipolar filtrations of~$\C$, denoted by $\{\mathcal{F}_n\}$ and $\{\mathcal{B}_n\}$ respectively.  As a consequence, this result cannot be reached by any combination of algebraic concordance invariants, Casson-Gordon invariants, and Heegaard-Floer invariants such as $\tau$, $\varepsilon$, and $\Upsilon$.  
We give two applications of this result. First, for any $n\ge 0$, there exists an infinite family $\{K^i\}_{i\geq 1}$ such that for each fixed $i$, $\{K^i_{2^j,1}\}_{j\geq 0}$ is a basis for an infinite rank summand of $\mathcal{F}_n$ and $\{K^i_{p,1}\}_{i, p\geq 1}$ is linearly independent in $\mathcal{F}_{n}/\mathcal{F}_{n.5}$. Second, for any $n\ge1$, we give filtered counterexamples to Kauffman's conjecture on slice knots by constructing smoothly slice knots with genus one Seifert surfaces where one derivative curve has nontrivial Arf invariant and the other is nontrivial in both $\F_n/\F_{n.5}$ and $\B_{n-1}/\B_{n+1}$. We also give examples of smoothly slice knots with genus one Seifert surfaces such that one derivative has nontrivial Arf invariant and the other is topologically slice but not smoothly slice.  
\end{abstract}
\maketitle

\section{Introduction}
Two knots are said to be (\emph{smoothly}) \emph{concordant} if they cobound a smooth annulus in~$S^3\times [0,1]$. The set of smooth concordance classes of knots forms an abelian group called the (\emph{smooth}) \emph{knot concordance group}, denoted by $\C$, under the connected sum operation. This group has been the subject of much study since its introduction by Fox and Milnor in~\cite{Fox-Milnor:1966-1}. A knot that is concordant to the unknot, or equivalently, bounds a smoothly embedded disk in $B^4$, is called a (\emph{smoothly}) \emph{slice} knot. There is a parallel, weaker notion of \emph{topological concordance} and \emph{topologically slice} knots, where the annuli and disks are required to be locally flat rather than smooth. 

Let $p,q$ be relatively prime integers and let $T_{p,q}$ denote the $(p,q)$ torus knot. Given a knot~$K$, the $(p,q)$ cable of $K$, denoted $K_{p,q}$, is obtained as the satellite of $K$ with pattern~$T_{p,q}$. In our notation, $K_{p,q}$ winds $p$ times in the longitudinal direction and $q$ times in the meridional direction. Observe that $K_{1,1}$ is isotopic to $K$. It is straightforward to see that cabling gives a well-defined function on $\C$ for any fixed $p,q$.  

From~\cite{Litherland:1979-1, Livingston-Melvin:1985-1}, we know that $\sigma_\omega(K_{p,q})=\sigma_\omega(T_{p,q})+\sigma_{\omega^p}(K)$ for any relatively prime integers $p$ and $q$, any knot $K$, and any $\omega$ on the unit circle away from roots of Alexander polynomials. For $\lvert p \rvert, \lvert q\rvert \geq 2$, this formula can be used to show that~$K$ and~$K_{p,q}$ are not concordant for any knot $K$. Indeed, $(p,q)$ cabling in general need not even preserve sliceness, since the nontrivial torus knots arise as cables of the unknot and are not slice. In contrast, $(p,1)$ cabling is more subtle. For instance, if $K$ is a slice knot, the knot~$K_{p,1}$ is slice for any $p$. That is, $K$ and $K_{p,1}$ are concordant. On the other hand, if $K$ has nontrivial signature function, for example, if $K$ is the right-handed trefoil, it is easy to see that~$\{K_{p,1}\}_{p\ge 1}$ is linearly independent. A similar conclusion may be drawn for many knots with non-vanishing Casson-Gordon sliceness obstructions~\cite{Kim:2005, Litherland:1984-1} or non-vanishing~$\Upsilon$-invariant~\cite{Ozsvath-Stipsicz-Szabo:2017-1, Feller-Park-Ray:2016-1,Chen:2016-1,Kim-Park:2016-1}. (For a variety of other results on the concordance of cables see also~\cite{Kawauchi:1980-1, Cochran-Franklin-Hedden-Horn:2013-1,Hedden:2009-1}.) This gives rise to the natural question: What can be said about the linear independence of $\{K_{p,1}\}_{p\ge 1}$ when $K$ is not slice, but all of these invariants vanish?  In this paper, we answer this question by producing infinite families of knots whose $(p,1)$ cables are linearly independent but which are too subtle to be detected by any of the tools mentioned above. More precisely, our examples will lie deep within certain filtrations of $\C$, which we now recall.  

The \emph{solvable filtration} of $\C$~\cite{Cochran-Orr-Teichner:2003-1}, 
\[
\cdots \subset \mathcal{F}_{n+1} \subset \mathcal{F}_{n.5} \subset \mathcal{F}_{n} \subset \cdots \subset \mathcal{F}_0 \subset \C,
\]
provides an infinite sequence of nontrivial sliceness obstructions~\cite{Cochran-Orr-Teichner:2003-1,Cochran-Orr-Teichner:2004-1,Cochran-Teichner:2007-1,Jiang:1981-1,Livingston:1999-2,Cochran-Harvey-Leidy:2009-1,Cochran-Harvey-Leidy:2009-03}. In particular, its lower levels encapsulate several classical concordance invariants. For example, a knot $K$ lies in $\F_0$ if and only if $\Arf(K)=0$; similarly, $K$ lies in $\F_{0.5}$ if and only if it is algebraically slice. Thus, any knot in $\F_{0.5}$ has vanishing signature function.  Additionally, every knot in $\F_{1.5}$ has vanishing Casson-Gordon sliceness obstructions.   
We remark in passing that if $K$ is topologically slice, $K\in \bigcap \F_n$ (however, the converse is open).  Moreover, there is an analogue of the solvable filtration, denoted $\{\Fn^\text{top}\}$, for the topological concordance group, and it is known that~$\Fn^\text{top}/\Fnpointfive^\text{top} \cong \Fn/\Fnpointfive$ for all $n$~\cite[p.\ 1423]{Cochran-Harvey-Leidy:2009-1}. Thus, the solvable filtration gives a language to organize knots for which topological concordance is increasingly subtle.  

In \cite{Cochran-Harvey-Horn:2013-1} Cochran-Harvey-Horn define a similar filtration,

\[
\cdots \subset \mathcal{B}_{n+1} \subset \mathcal{B}_{n}  \subset \cdots \subset \mathcal{B}_0 \subset \C,
\] called the \emph{bipolar filtration}, specifically geared towards the study of the smooth knot concordance group. This filtration has proved particularly useful in the study of smooth concordance classes of topologically slice knots~\cite{Cochran-Horn:2015-1, Cha-Powell:2014-1, Cha-Kim:2017-1}. For the smooth concordance group, it was shown in~\cite{Cochran-Harvey-Horn:2013-1} that the bipolar filtration provides an infinite sequence of nontrivial sliceness obstructions, and in particular, several strong concordance invariants, including the Heegaard-Floer invariants $\tau$~\cite{Ozsvath-Szabo:2003-1} and $\varepsilon$~\cite{Hom:2014-1}, as well as Rasmussen's~$s$-invariant~\cite{Rasmussen:2010}, vanish on $\B_0$. Further, we know from~\cite{Cochran-Harvey-Horn:2013-1, Ozsvath-Stipsicz-Szabo:2017-1, Ni-Wu:2015-1, Hom-Wu:2016-1} that the $\nu^+$-invariant~\cite{Hom-Wu:2016-1} and $\Upsilon$-invariant~\cite{Ozsvath-Stipsicz-Szabo:2017-1} also vanish on $\mathcal{B}_0$. Thus, the bipolar  filtration gives a language to organize knots for which smooth concordance is increasingly subtle. 

In the present paper, we study the effect of cabling on $\C$ through the lens of the solvable and bipolar filtrations.  It is easy to see that $(p,1)$ cabling is a well-defined operation on~$\F_n$ and~$\B_n$ for any $n$ and $p$ (Proposition~\ref{prop:cablefiltration}). The following is our main result. 

\newtheorem*{thm:cables-indep}{Theorem~\ref{thm:cables-indep}}
\begin{thmx}~\label{thm:cables-indep}
For any $n\ge1$, there exists an infinite family of knots $\{K^i\}_{i\geq 1}\subset \F_n\cap \B_{n-1}$, such that the set of cables $\{K^i_{p,1}\}_{i,p\geq 1}$ is linearly independent in $\Fn/\Fnpointfive$ and in $\B_{n-1}/\B_{n+1}$.
\end{thmx}

In particular, the examples given above which lie in $\F_1$ have vanishing Levine-Tristram signature function, those in $\F_2$ have vanishing Casson-Gordon sliceness obstructions, and those in $\B_0$ have vanishing $\tau$-, $\varepsilon$-, and $\Upsilon$-invariants. Consequently, Theorem~\ref{thm:cables-indep} shows that the above tools are not sufficient to completely detect the linear independence of cables in~$\C$.

The proof of Theorem~\ref{thm:cables-indep} uses the technology of~\cite{Cochran-Harvey-Leidy:2011-1}, namely \emph{robust doubling operators} (Definition~\ref{def:robust}) and von Neumann $\rho$-invariants (Section~\ref{sec:L2-signature-invariants}). The main technical result shows that for any $p$, the $(p,1)$ cable of a robust doubling operator is also a robust doubling operator, under a mild hypothesis on the Alexander polynomial (Theorem~\ref{prop:cablerobust}).

As an application of our main result, we show that certain families of $(p,1)$ cables form bases for infinite rank summands of $\Fn$. 

\newtheorem*{thm:summand}{Theorem~\ref{thm:summand}}
\begin{thmx}~\label{thm:summand} For any $n\geq 0$, there exists an infinite family of knots $\{K^i\}_{i\ge 1}\subset  \F_n$ such that, for each fixed $i$, the set $\{K^i_{2^j,1}\}_{j\geq 0}$ is a basis for an infinite rank summand of $\Fn$ and for which $\{K^i_{p,1}\}_{i, p\geq 1}$ is linearly independent in $\mathcal{F}_{n}/\mathcal{F}_{n.5}$.
\end{thmx}

The proof of the above theorem uses a result of Feller and the second and third authors~\cite{Feller-Park-Ray:2016-1}, using the $\Upsilon$-invariant, which says that for any genus one knot $K$ with $\tau(K)=1$, the set $\{K_{2^j,1}\}_{j\geq 0}$ is a basis for an infinite rank summand of $\C$.

As a second application of Theorem~\ref{thm:cables-indep}, we consider Kauffman's conjecture on slice knots~\cite[p.~226]{Kauffman:1987-1}. For a knot $K$ with a genus $g$ Seifert surface $F$, a $g$-component link $d$ on $F$ is called a \emph{derivative for $K$ associated with $F$} if it is homologically essential and the linking form vanishes on it. Note that if $d$ is slice, so is $K$ since we may obtain a slice disk for $K$ by performing ambient surgery on $F$ using a collection of slice disks for $d$. If $F$ has genus one, $d$ is simply a knot. Elementary linear algebra shows that a genus one Seifert surface for an algebraically slice knot $K$ has exactly two derivatives, up to orientation. In the 1980s, Kauffman conjectured that if a knot $K$ is slice, then any genus one Seifert surface for $K$ will have a derivative knot $d$ such that $d$ is slice, or at least $\Arf(d)=0$. In~\cite{Cochran-Davis:2015-1}, Cochran and the first author constructed counterexamples to Kauffman's conjecture, by finding slice knots with genus one Seifert surfaces where one derivative has nontrivial Arf invariant and the other is of the form $L\#-L_{2,1}$ where $L$ is any knot and $-L_{2,1}$ denotes the reverse of the mirror image of $L_{2,1}$. By judicious choice of $L$, they ensured that $\Arf(L\# -L_{2,1})=1$, i.e.\ each derivative is nontrivial in $\C/\F_0$. We give a filtered version of their result, as follows. 

\newtheorem*{thm:kauffman}{Theorem~\ref{thm:kauffman}}
\begin{thmx}~\label{thm:kauffman}
For any $n\geq 1$, there exists a slice knot $K$ bounding a genus one Seifert surface with derivative curves $d$ and $d'$ such that $\Arf(d')\neq 0$ and $d$ is nontrivial in each of the quotients $\Fn / \Fnpointfive$ and $\B_{n-1}/\B_{n+1}$. 
\end{thmx}

To complete the set of examples, we also show that there exist slice knots with derivatives that are topologically slice but not slice. Since all topologically slice knots lie in $\bigcap \F_n$, this is an example of a non-slice derivative in $\bigcap \F_n$. 

\newtheorem*{thm:kauffmanprime}{Theorem~\ref{thm:kauffmanprime}}
\stepcounter{thmxprime}
\stepcounter{thmxprime}
\begin{thmxprime}\label{thm:kauffmanprime}
There exists a slice knot $K$ bounding a genus one Seifert surface with derivative curves $d$ and $d'$, such that $\Arf(d')\neq 0$ and $d$ is topologically slice but not slice.
\end{thmxprime}

Let $K$ be a knot with a genus one Seifert surface $F$. Then $F \times I\subset S^3\times I$ is a genus two handlebody. As we noted above, if a derivative curve $d$ on $F$ bounds a topological (resp.\ smooth) slice disk $\Delta_d$ in $B^4$, we can do ambient surgery on $F\times \{1\}$ using the slice disk for $d\times \{1\}\subset S^3\times \{1\}$ to yield a topological (resp.\ smooth) slice disk for $K$; call this new slice disk $\Delta$. 
By gluing a thickened copy of $\Delta_d$ to $F\times I$, we see that the genus one surface~$F \cup \Delta$ bounds a topological (resp.\ smooth) handlebody in $B^4$. Conversely, if there exists some slice disk $\Delta$ for $K$ such that $F\cup \Delta$ bounds a topological (resp.\ smooth) handlebody, then there must be some derivative curve on $F$ which is topologically (resp.\ smoothly) slice. Thus, Theorem~\ref{thm:kauffmanprime} gives a smoothly slice knot $K$ with a genus one Seifert surface $F$ such that there exists a topological slice disk $\Delta_{\text{top}}$ for $K$ such that $F \cup \Delta_{\text{top}}$ bounds a topological handlebody in $B^4$, but $F\cup \Delta$ does not bound a smooth handlebody in $B^4$ for any smooth slice disk $\Delta$ for $K$.  This is done explicitly in Corollary \ref{cor:handlebody}.

\subsection*{Outline}
In Section~\ref{sec:background} we give some background on von Neumann $\rho$-invariants and the solvable and bipolar filtrations of $\C$. Section~\ref{sec:robust} introduces robust doubling operators and we prove the key technical result that $(p,1)$ cables of robust doubling operators are often robust (Theorem \ref{prop:cablerobust}). Sections~\ref{sec:proofofA},~\ref{sec:proofofB}, and~\ref{sec:proofofC} give the proofs of Theorems~\ref{thm:cables-indep},~\ref{thm:summand}, and~\ref{thm:kauffman} respectively. 

\subsection*{Notation and conventions} 
For any knot $K$, $-K$ denotes the mirror image of $K$ with the circle orientation reversed and $M(K)$ refers to the $0$-framed surgery on $S^3$ along $K$. We denote the set of non-negative integers by $\mathbb{N}_0$ and the set of positive integers by $\mathbb{N}$. Throughout the paper, smoothness should be assumed whenever we say concordant, concordance, or slice without qualification. We also assume that any topological embedding is locally flat.

\subsection*{Acknowledgements} Much of this research was completed while the authors participated in a trimester on Topology at the Hausdorff Institute for Mathematics in Bonn, Germany in the fall of 2016. We also worked on this project while the first author visited the second and third at the Max-Planck-Institut f\"{u}r Mathematik in Bonn, Germany in the fall of 2017 and the summer of 2018. We thank both HIM and MPIM for a stimulating research atmosphere. We also thank Jae Choon Cha and Danny Ruberman for helpful conversations.\\

\section{Background}~\label{sec:background}

\subsection{Von Neumann $\rho$-invariants}\label{sec:L2-signature-invariants}
Let $M$ be a closed, oriented $3$-manifold, $\Gamma$ be a discrete group, and $\phi \colon \pi_1(M) \rightarrow \Gamma$ be a representation. The \emph{von Neumann $\rho$-invariant}, $\rho(M,\phi)\in\mathbb{R}$, was defined by Cheeger and Gromov in \cite{Cheeger-Gromov:1985-1}. For a detailed discussion from our perspective, see \cite[Section~5]{Cochran-Harvey-Leidy:2011-1},\cite[Section~2]{Cochran-Teichner:2007-1}, or \cite[Section~2]{Cochran-Orr-Teichner:2004-1}. We will only need the following properties\label{rho-properties}. 
\begin{enumerate}
\item If $M=M_1\sqcup M_2$ with $\phi=\phi_1\sqcup \phi_2$ where $\phi_1\colon \pi_1(M_1)\to\Gamma$ and $\phi_2\colon \pi_1(M_2)\to\Gamma$, then $\rho(M,\phi)=\rho(M_1,\phi_1)+\rho(M_2,\phi_2)$.
\item If $-M$ denotes the orientation reverse of $M$, then $\rho(-M,\phi)=-\rho(M,\phi)$. 
\item If $(M,\phi) = \partial(W, \psi)$ for some compact, oriented $4$-manifold $W$ and $\psi\colon \pi_1(W)\to \Gamma$, then $\rho(M,\phi) = \sigma^{(2)}(W,\psi) -\sigma(W)$ where $\sigma^{(2)}(W,\psi)$ is the $L^2$-signature of the equivariant intersection form  on 
$H_2(W;\Z[\Gamma])$
\item If $\phi$ factors through  $\phi'\colon \pi_1(M)\to\Gamma'$ where $\Gamma'\leq \Gamma$, then $\rho(M,\phi')=\rho(M,\phi)$. 
\end{enumerate}

For any knot $K$, recall that $M(K)$ refers to the $0$-framed surgery on $S^3$ along $K$.  The von Neumann $\rho$-invariants corresponding to certain special representations of $\pi_1(M(K))$ will be important to us. First, for a knot $K$, we define $\rho_0(K):= \rho(M(K),\phi)$, where 
\[
\phi \colon \pi_1(M(K)) \rightarrow \mathbb{Z}
\]
is the abelianization map. It is known that $\rho_0(K)$ equals the integral of the Levine-Tristram signature function of $K$ over $S^1$~\cite[Proposition~5.1]{Cochran-Orr-Teichner:2004-1}. 

Let $G$ be the group $\pi_1(M(K))$. Since $G/G^{(1)}\cong \Z$, the exact sequence 
$$
0\to G^{(1)}/G^{(2)}\to G/G^{(2)}\to G/G^{(1)}\to 0
$$
splits and so $G/G^{(2)}\cong \Z\ltimes G^{(1)}/G^{(2)}$.    Let $\A(K)$ denote the rational Alexander module of~$K$.  As $G^{(1)}/G^{(2)}$ is $\Z$-torsion free, $G^{(1)}/G^{(2)}\into \A(K)$. 
For any submodule $P \subseteq \mathcal{A}(K)$, 
we obtain the associated homomorphism
\[
\phi_P\colon G\to G/G^{(2)}\cong \Z\ltimes G^{(1)}/G^{(2)}\into \Z\ltimes \A(K)\to  \Z\ltimes \A(K)/P
\]
and thus, the corresponding von Neumann $\rho$-invariant, $\rho(M(K),\phi_P)$. 

\begin{remark}\label{rem:alex-module-defn}
For any knot $K$, we have the inclusion-induced isomorphism 
\[
H_1(E(K);\Q[t,t^{-1}])\xrightarrow{\cong} H_1(M(K);\Q[t,t^{-1}]),
\]
where $E(K)$ is the exterior of $K$ in $S^3$. As a result, either of the two above modules may be taken to be the rational Alexander module of $K$. As is common practice, we will switch back and forth between these two definitions when convenient and without comment.
\end{remark}

Recall that there exists a form $\Bl:\mathcal{A}(K)\times \mathcal{A}(K)\to \Q(t)/\Q[t,t^{-1}]$ called the \emph{Blanchfield form}. A submodule $P\subseteq \mathcal{A}(K)$ is said to be \emph{isotropic} if $P\subseteq P^\perp$ and \emph{Lagrangian} if $P=P^\perp$, where $P^\perp := \{q\in \mathcal{A}(K) \mid \Bl(p,q)=0 \text{ for all } p\in P\}$.

\begin{definition}[{\cite[Section~4]{Cochran-Harvey-Leidy:2008-1}},{\cite[Section~3]{Cochran-Harvey-Leidy:2009-1}}]The \emph{first-order signatures} of a knot $K$ are the real numbers $\rho(M(K),\phi_P)$ where $P$ is any isotropic submodule of $\mathcal{A}(K)$ with respect to the Blanchfield form on $\A(K)$. For any knot $K$, the set of all first-order signatures of $K$ will be denoted $\mathcal{FOS}(K)$.
\end{definition}

\begin{remark}\label{rem:1stsgn-defn} Note that $\phi_P$ and $\rho(M(K),\phi_P)$ are defined in a slightly different way elsewhere in the literature, such as in~\cite{Cochran-Harvey-Leidy:2008-1,Cochran-Harvey-Leidy:2009-1}. For any submodule $P \subseteq \mathcal{A}(K)$, let 
\[
\phi'_P\colon G\to G/\widetilde{P},
\]
where 
$$\widetilde{P} = \{x\mid x \in \ker(G^{(1)} \to G^{(1)}/G^{(2)} \to \A(K)/P)\}.$$ Since $\widetilde{P} = \ker \phi_P$, we see that $\phi_P$ factors through~$\phi'_P\colon G\to G/\widetilde{P}$, where~$G/\widetilde{P} \leq \Z\ltimes \A(K)/P$. Then by property $(4)$ for von Neumann $\rho$-invariants given above, $\rho(M(K),\phi_P)=\rho(M(K),\phi'_P)$. This justifies the slight alteration in the definitions.
\end{remark}

If a knot $K$ bounds a slice disk $\Delta \subset B^4$, we have the following submodule: 
\[
P_\Delta := \ker \left(H_1(M(K);\mathbb{Q}[t,t^{-1}])\rightarrow H_1(B^4 \setminus \Delta; \mathbb{Q}[t,t^{-1}])\right).
\]
By~\cite[Corollary 2]{Kearton:1975-2}, $P_\Delta$ is a Lagrangian submodule of $\mathcal{A}(K)$ with respect to the  Blanchfield form and we say that the Lagrangian $P_\Delta$ \emph{corresponds to the slice disk $\Delta$}. By \cite[Theorem 4.2]{Cochran-Orr-Teichner:2003-1}, $\rho(M(K), \phi_{P_\Delta}) = 0$.

\begin{remark}~\label{rmk:metalag}
If $d$ is a derivative on a Seifert surface $F$ for a knot $R$, let $\langle d \rangle$ be the submodule of $\A(R)$ generated by a lift of $d$ in the infinite cyclic cover of $S^3 \setminus R$. It is well-known that $\langle d \rangle$ is a  Lagrangian submodule of $\A(R)$ with respect to the  Blanchfield form (see e.g.\ \cite[Theorem~2]{Kearton:1975-2}\cite[Lemma~5.9]{Davis:2012-2}). We say that $\langle d \rangle$ is the \emph{Lagrangian generated by the derivative $d$}. If $d$ is slice, then $\langle d \rangle$ is in fact the Lagrangian of $\A(K)$ corresponding to the slice disk for $R$ obtained by performing ambient surgery on $F$ using a collection of slice disks for $d$.\end{remark}

\subsection{Filtrations of $\C$}

\begin{definition}[{\cite{Cochran-Orr-Teichner:2003-1}}]\label{defn:n-solvable}
A knot $K$ is said to be \emph{$n$-solvable} for $n \in \N_0$ if $K$ bounds a smooth, properly embedded disk $\Delta$ in a smooth, compact, oriented $4$-manifold $W$ with~$\partial W = S^3$ and $H_1(W;\Z)=0$ such  that $H_2(W;\Z)$ has a basis consisting of $2k$ smoothly embedded, connected, compact, oriented surfaces $\{L_1,\dots ,L_k,D_1,\dots,D_k\}$ in $W\setminus\Delta$, for some $k$, with trivial normal bundles, satisfying:
\begin{enumerate}
\item  $\pi_1(L_i) \subset \pi_1(W\setminus \Delta)^{(n)}$ and $\pi_1(D_i) \subset \pi_1(W\setminus \Delta)^{(n)}$ under inclusion for all $i=1,\dots,k$, and 
\item the geometric intersection numbers are $L_i \cdot L_j = 0 = D_i\cdot D_j$ and $L_i \cdot D_j = \delta_{ij}$ for all $i,j =1,\dots,k$.
\end{enumerate}
Such a $4$-manifold~$W$ is called an \emph{$n$-solution} for $K$. 

If, in addition, $\pi_1(L_i) \subset \pi_1(W)^{(n+1)}$ for all $i$, then we say that $K$ is \emph{$(n.5)$-solvable} and~$W$ is an \emph{$(n.5)$-solution} for $K$. 

The collection of $n$-solvable (resp.\ $(n.5)$-solvable) knots forms a subgroup of $\C$, denoted by $\mathcal{F}_{n}$ (resp.\ $\mathcal{F}_{n.5}$).  
\end{definition}

Above, recall that for any group $G$ and $n\geq 0$, $G^{(n)}$ denotes the $n$th term of the \emph{derived series} for $G$, that is, $G^{(0)}=G$ and $G^{(n+1)}=[G^{(n)},G^{(n)}]$.

In~\cite{Cochran-Harvey-Horn:2013-1}, three new filtrations of $\C$ were defined as follows. Note that the main difference is in the intersection form of the $4$-manifolds being considered. 

\begin{definition}[{\cite{Cochran-Harvey-Horn:2013-1}}]~\label{defn:bipolar}
A knot $K$ is said to be \emph{$n$-positive} (resp.\ \emph{$n$-negative}) for $n\in\mathbb{N}_0$ if $K$ bounds a smooth, properly embedded disk $\Delta$ in a smooth, compact, oriented $4$-manifold $W$ with $\partial W = S^3$ and $\pi_1(W)=1$ such that $H_2(W;\Z)$ has a basis consisting of smoothly embedded, connected, compact, oriented surfaces $\{S_1,\dots, S_k\}$ in $W\setminus \Delta$, for some~$k$, satisfying:
\begin{enumerate}
\item  $\pi_1(S_i) \subset \pi_1(W\setminus \Delta)^{(n)}$ under inclusion for each $i=1,\dots, k$, and
\item the intersection form on $H_2(W;\Z)$ is positive definite (resp.\ negative definite).
\end{enumerate}
Such a $4$-manifold $W$ is called an \emph{$n$-positon} (resp.\ \emph{$n$-negaton}) for $K$. 

The collection of $n$-positive (resp.\ $n$-negative) knots forms a submonoid of $\C$, denoted by~$\mathcal{P}_{n}$ (resp.\ $\mathcal{N}_{n}$). 

We say a knot $K$ is \emph{$n$-bipolar} if $K$ is both $n$-positive and $n$-negative. The collection of $n$-bipolar knots forms a subgroup of $\mathcal{C}$, denoted by $\mathcal{B}_{n}$.
\end{definition}

\begin{remark}\label{rem:rho-obstructions}
In~\cite[Corollary~5.6]{Cochran-Harvey-Horn:2013-1}, the following relationships were noted.
\[
\Bn\subset \langle \Pn\rangle=\langle \Nn\rangle = \langle \Pn\cup \Nn\rangle \subseteq \Fn^\mathbb{Q}
\]
where $\{\Fn^\mathbb{Q}\}$ denotes the filtration of $\C$ obtained using the same definition as Definition~\ref{defn:n-solvable}, except that all occurrences of $\Z$ are replaced by $\Q$. Importantly, it is known that all the obstructions to membership in $\Fnpointfive$ arising from von Neumann $\rho$-invariants also obstruct membership in $\Fnpointfive^\Q$~\cite[p.\ 2123]{Cochran-Harvey-Horn:2013-1}\cite[Section~7]{Cochran-Harvey-Horn:2013-1}\cite[Section~4]{Cochran-Orr-Teichner:2003-1}, and thus, obstruct membership in $\B_{n+1}$.
\end{remark}

We end this section by stating an elementary construction of $0$-positive and $0$-negative knots.

\begin{proposition}[{\cite[Proposition~3.1]{Cochran-Harvey-Horn:2013-1}}]~\label{prop:postive} Any knot that can be transformed to a knot in $\mathcal{P}_0$ by changing positive crossings to negative crossings lies in $\mathcal{P}_0$. Similarly, any knot that can be transformed to a knot in $\mathcal{N}_0$ by changing negative crossings to positive crossings lies in $\mathcal{N}_0$.
\end{proposition}

\subsection{Infection on knots}
We will need to construct elements of $\Fn$ and $\Bn$ for any $n$. These examples will be built using the process of iterated \emph{infection} on ribbon knots, which we now describe. Let $R$ be a fixed knot in $S^3$ and $\alpha$ be a curve in $S^3\setminus R$ which is unknotted in $S^3$. The complement $S^3\setminus \alpha$ is a solid torus. Given any knot $K$, the knot~$R_\alpha(K)$, the result of \emph{infection on $R$ with $K$ along $\alpha$}, is the image of $R$ in the copy of $S^3$ obtained by gluing the solid torus $S^3\setminus \alpha$ to $S^3\setminus K$ by sending the meridian of $\alpha$ to the longitude of $K$ and the longitude of $\alpha$ to the meridian of $K$, as shown in Figure~\ref{fig:infection}. The reader should note that this is simply a different perspective on the classical (untwisted) satellite operation, where~$R\subset S^3\setminus \alpha$ is the pattern and $K$ is the companion. Pictorially, we can imagine cutting the strands of $R$ piercing through the disk bounded by $\alpha$, tying them into the knot~$K$ (with no twisting), and then gluing the strands back together. It is easy to see that a fixed knot~$R$ along with a choice of unknotted curve $\alpha\subset S^3\setminus R$ yields a function~$R_\alpha:\C\to\C$, mapping~$K\mapsto R_\alpha(K)$. We will denote such a function by the link~$(R,\alpha)$ as in the left panel of Figure~\ref{fig:infection}. 

\begin{figure}[htb]
\labellist
\small\hair 2pt
\pinlabel $\alpha$ at 92 30
\pinlabel $R$ at 60 -10
\pinlabel $K$ at 225 65
\pinlabel $K$ at 353 30
\pinlabel $R_\alpha(K)$ at 353 -10
\endlabellist
\centering
\includegraphics{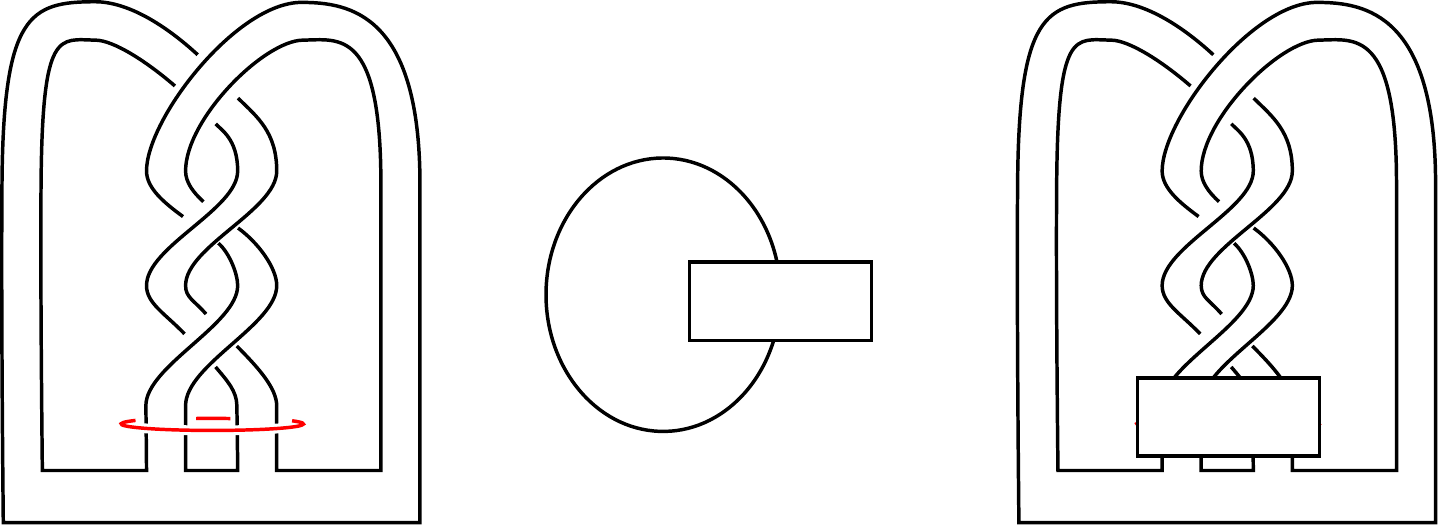}
\vspace*{5mm}
\caption{Left: A knot $R$ along with an unknotted curve $\alpha\subset S^3\setminus R$. The link $(R,\alpha)$ shows a doubling operator. Middle: A knot $K$. Right: The result of infection on $R$ with $K$ along $\alpha$. Any solid box containing $K$ indicates that all the strands passing through the box are tied into $0$-framed parallels of the knot $K$.}\label{fig:infection}
\end{figure}
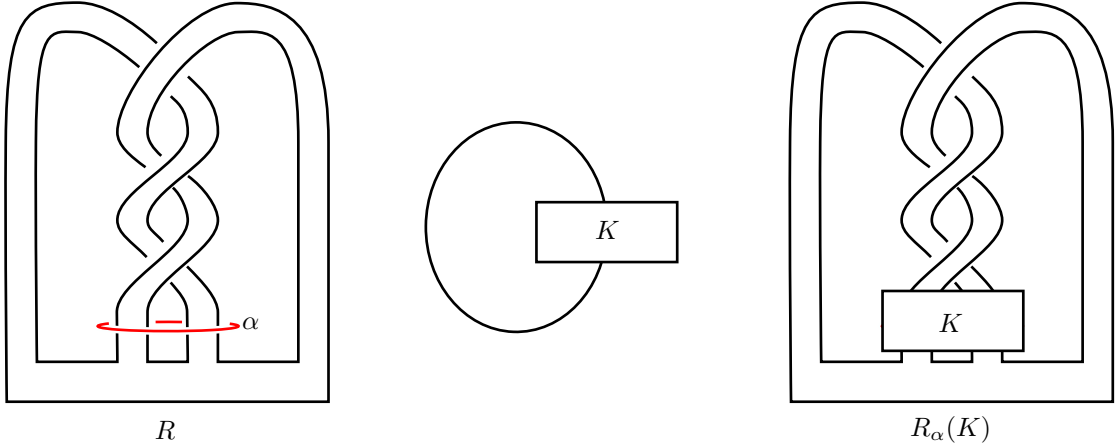

\begin{definition}A \emph{doubling operator} is a function  $R_{\alpha} : \mathcal{C} \rightarrow \mathcal{C}$ arising from infection on a ribbon knot $R$ along an unknotted curve $\alpha\subset S^3\setminus R$ where $\lk(R, \alpha) = 0$.
\end{definition}

Doubling operators enable us to build knots deep in the various filtrations of $\C$ by the following result. 

\begin{proposition}[{\cite[Lemma~6.4]{Cochran-Harvey-Leidy:2008-1}}{\cite[Proposition~3.3]{Cochran-Harvey-Horn:2013-1}}]\label{prop:infection}
Let $R_\alpha$ be a doubling operator.
\begin{enumerate}
\item If $K\in\Fn$ for some $n\in\N_0$, then $R_\alpha(K)\in\F_{n+1}$;
\item if $K\in\Pn$ for some $n\in\N_0$, then $R_\alpha(K)\in\mathcal{P}_{n+1}$;
\item if $K\in\Nn$ for some $n\in\N_0$, then $R_\alpha(K)\in\mathcal{N}_{n+1}$; and
\item if $K\in\Bn$ for some $n\in\N_0$, then $R_\alpha(K)\in\B_{n+1}$.
\end{enumerate}
\end{proposition}

Additionally, it is straightforward to show that the $(p,1)$ cables of a knot in $\F_n$ (resp.\ $\Bn$) lie in $\Fn$ (resp.\ $\Bn$) for all $n$ and $p$, as follows.
   
\begin{proposition}\label{prop:cablefiltration}
Let $n\geq 0$ and $p\geq 1$. If $K\in\Fn$, then~$K_{p,1}\in\F_n$; if $K\in\Bn$, then~$K_{p,1}\in\B_n$.
\end{proposition}

\begin{proof}
By hypothesis $K$ bounds a properly embedded disk $\Delta\subset W$ where $W$ is an $n$-solution for $K$. We will show that $W$ is also an $n$-solution for $K_{p,1}$. By hypothesis, there exist smoothly embedded surfaces $\{L_1,\dots, L_k, D_1, \dots, D_k\}$ within $W\setminus N(\Delta)$ which form a basis for $H_2(W)$, such that $\pi_1(L_i)\subset \pi_1(W\setminus N(\Delta))^{(n)}$ and $\pi_1(D_i)\subset \pi_1(W\setminus N(\Delta))^{(n)}$ for all $i$, where $N(\Delta)$ is a tubular neighborhood of $\Delta$. Let $\Delta_{p,1}$ be the disk in $W$ bounded by $K_{p,1}$ obtained by taking $p$ parallel copies of $\Delta$ and joining them together by bands, within $N(\Delta)$. It only remains to show that $\pi_1(L_i)\subset \pi_1(W\setminus N(\Delta_{p,1}))^{(n)}$ and $\pi_1(D_i)\subset \pi_1(W\setminus N(\Delta_{p,1}))^{(n)}$ for all $i$, which follows from the functoriality of the derived series since~$W\setminus N(\Delta)\subset W\setminus N(\Delta_{p,1})$. An identical proof for $n$-positons and $n$-negatons gives the second statement.
\end{proof}

We end with the following proposition regarding the first-order signatures of knots obtained by infection.
\begin{proposition}[{\cite[Lemma~2.3]{Cochran-Harvey-Leidy:2009-1}}]\label{prop:rho-infection}
For any doubling operator $R_\alpha$ and knot $K$, there is an isomorphism $i_*:\A(R_\alpha(K))\to \A(R)$ such that for every submodule $P\subseteq \A(R_\alpha(K))$, 
\[
\rho(R_\alpha(K), \phi_{P})=
 \begin{cases} 
      \rho(R, \phi_{i_*(P)})+\rho_0(K) & \text{if }[\alpha]\notin i_*(P) \\
      \rho(R, \phi_{i_*(P)}) & \text{if }[\alpha]\in i_*(P).
   \end{cases}
\]
\end{proposition}

\subsection{Strong coprimality of sequences of Laurent polynomials}

Recall that an element $p(t)$ of $\Q[t,t^{-1}]$ is said to be \emph{prime} if it is non-zero, not a unit, and whenever~$p(t)$ divides $a(t)b(t)$ for some $a(t), b(t)\in\Q[t,t^{-1}]$, then either $p(t)$ divides $a(t)$ or $p(t)$ divides~$b(t)$. Note that $\Q[t,t^{-1}]$ is a Euclidean domain, where the size function is given by the difference in the degrees of the lowest and highest terms in a given Laurent polynomial. Thus, an element in $\Q[t,t^{-1}]$ is prime if and only if it is irreducible. Recall that the units in~$\Q[t,t^{-1}]$ are exactly the non-zero monomials. As usual, two non-zero elements $p(t)$ and~$q(t)$ of $\Q[t,t^{-1}]$ are said to be \emph{coprime}, denoted by $(p,q)=1$, if their greatest common divisor is a unit in $\Q[t,t^{-1}]$. 

We now give a notion of strong primality and strong coprimality of polynomials, and more generally, sequences of polynomials.

\begin{definition}{{\cite[Definition~4.4 and Definition~6.1]{Cochran-Harvey-Leidy:2011-1}}} A 
non-zero element $p(t)\in\mathbb{Q}[t,t^{-1}]$ is said to be \emph{strongly prime} or \emph{strongly irreducible} if $p(t^k)$ is prime (equivalently, irreducible) in $\mathbb{Q}[t, t^{-1}]$ for every non-zero integer $k$.

Two non-zero elements $p(t), q(t) \in \mathbb{Q}[t, t^{-1}]$ are said to be \emph{strongly coprime}, denoted $\widetilde{(p, q)} = 1$, if, for every pair of non-zero integers $k$ and $\ell$, $p(t^k)$ and $q(t^\ell)$ are coprime in~$\mathbb{Q}[t, t^{-1}]$. 

Given $\mathcal{P} = (p_1(t), \ldots , p_n(t))$ and $\mathcal{Q} = (q_1(t), \ldots , q_n(t))$, two sequences of non-zero elements in $\mathbb{Q}[t,t^{-1}]$, we say that $\mathcal{P}$ is \emph{strongly coprime} to $\mathcal{Q}$ if either $(p_1, q_1) = 1$, or, for some~$i > 1$, $\widetilde{(p_i , q_i )} = 1$.
\end{definition}

\begin{remark}\label{rem:primality}
 Note that if a non-monomial polynomial $p(t)$ is prime in $\Q[t]$, then it is prime in $\Q[t,t^{-1}]$. Since $\Q[t]$ is a principal ideal domain, $p(t)$ is prime in $\Q[t]$ if and only if it is irreducible in $\Q[t]$. It follows from Gauss's Lemma (see for example \cite[Section~9.3, Proposition~5]{DummitFoote}) that if a polynomial with integer coefficients is irreducible in $\Z[t]$ then it is irreducible in $\Q[t]$.  As a result, if a non-monomial polynomial with integer coefficients is irreducible in $\Z[t]$, then it is prime in $\Q[t,t^{-1}]$.
\end{remark}

\section{Robust doubling operators and cabling}\label{sec:robust}

In~\cite{Cochran-Harvey-Leidy:2011-1}, Cochran-Harvey-Leidy show that knots obtained by $n$-fold iterated infection by certain \emph{robust doubling operators} are linearly independent in $\Fn/\Fnpointfive$. The goal of this section is to show that given a robust doubling operator, the operator given by its~$(p,1)$ cable is also robust, under a mild condition on the Alexander polynomials related to the integer~$p$.   We begin by recalling the definition of a robust doubling operator.  

\begin{definition}[{\cite[Definition~7.2]{Cochran-Harvey-Leidy:2011-1}}]\label{def:robust} A doubling operator $R_\alpha \colon \mathcal{C} \rightarrow \mathcal{C}$ is said to be \emph{robust} if 
\begin{enumerate}
\item $\mathcal{A}(R)$ is generated by $[\alpha]$, with 
\[
\mathcal{A}(R) \cong \frac{\mathbb{Q}[t,t^{-1}]}{\langle\delta(t)\cdot\delta(t^{-1})\rangle}
\]
where $\delta(t)$ is prime in $\Q[t,t^{-1}]$, and
\item for each isotropic submodule $P\subseteq A(R)$, either the first-order signature $\rho(M(R),\phi_P)$ is non-zero or $P$ corresponds to a ribbon disk for $R$.
\end{enumerate}
\end{definition}

\begin{remark}\label{rem:3submodules}
Note that the Alexander module of a knot $R$ with~$\mathcal{A}(R) \cong \frac{\mathbb{Q}[t,t^{-1}]}{\langle\delta(t)\cdot\delta(t^{-1})\rangle}$ with~$\delta(t)$ prime in $\Q[t,t^{-1}]$ has at most three proper submodules, namely $\langle 0\rangle$, $\langle \delta(t)\rangle$, and~$\langle \delta(t^{-1})\rangle$. This follows since $\delta(t)$ is prime in $\Q[t,t^{-1}]$ iff $\delta(t^{-1})$ is. Note that the submodules $\langle \delta(t)\rangle$ and $\langle \delta(t^{-1})\rangle$ coincide if $\delta(t)$ and $\delta(t^{-1})$ differ by multiplication by a unit but otherwise, have trivial intersection.

If $R$ is ribbon, it follows that one of the nontrivial proper submodules $\langle \delta(t)\rangle$ and $\langle \delta(t^{-1})\rangle$ must be a Lagrangian which corresponds to a ribbon disk. It is straightforward to check using the properties of the Blanchfield form that the other nontrivial proper submodule is also a Lagrangian. Thus, 
\[
\mathcal{FOS}(R)=\left\{\rho\left(M(R),\phi_{\langle 0\rangle}\right),\rho\left(M(R),\phi_{\langle \delta(t)\rangle}\right),\rho\left(M(R),\phi_{\langle \delta(t^{-1})\rangle}\right)\right\}.
\]
\end{remark}

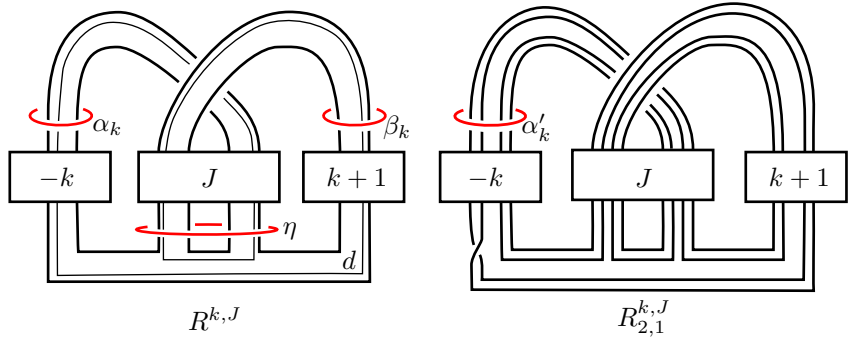
\begin{figure}[htb]
\labellist
\small\hair 2pt
\pinlabel $R^{k,J}$ at 77 -10
\pinlabel $J$ at 75 43
\pinlabel $\beta_k$ at 145 62
\pinlabel $\alpha_k$ at 37 62
\pinlabel $-k$ at 18 43
\pinlabel $k+1$ at 130 43
\pinlabel $\eta$ at 105 23
\pinlabel $d$ at 127 12

\pinlabel $J$ at 237 43
\pinlabel $\alpha'_k$ at 197 62
\pinlabel $-k$ at 178 43
\pinlabel $k+1$ at 295 43

\pinlabel $R^{k,J}_{2,1}$ at 237 -10
\endlabellist
\centering
\includegraphics{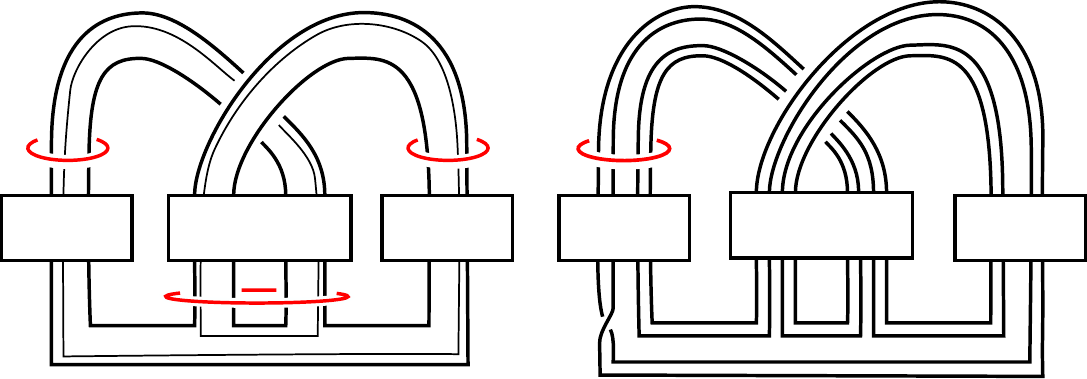}
\vspace*{5mm}
\caption{Each solid box containing an integer indicates the number of full right-handed twists among all the strands passing through the box. Each solid box containing $J$ indicates that all the strands passing through the box are tied into $0$-framed parallels of the knot $J$.  Left: The doubling operator $R_{\alpha_k}^{k,J}$. Right: The doubling operator $\left(R^{k,J}_{2,1}\right)_{\alpha_k'}$. }\label{fig:robust-with-ambiguity}
\end{figure}

Let $R_{\alpha_k}^{k,J}$ denote the doubling operator on the left hand side of Figure~\ref{fig:robust-with-ambiguity}. A direct computation shows that the Alexander polynomial of $R^{k,J}$ is 
\[
\Delta_{R^{k,J}}(t) = \left(kt-(k+1)\right)\left((k+1)t-k\right).
\] 
We now show that these doubling operators are robust for many choices of $J$ and $k\geq 1$. These are the operators we use for the proof of Theorem~\ref{thm:cables-indep}.

\begin{proposition}~\label{prop:robust-with-ambiguity} 
For $k\geq 1$, and any knot $J$ with $-\rho_0(J)\notin \mathcal{FOS}(R^{k,U})$, the doubling operator $R_{\alpha_k}^{k,J}$ is robust.
\end{proposition}

Since $\mathcal{FOS}(R^{k,U})$ is a finite set, $J$ can be taken to be the connected sum of an appropriately large number of knots with nontrivial $\rho_0$, for example, the connected sum of many right-handed trefoils. 

\begin{proof}[Proof of Proposition~\ref{prop:robust-with-ambiguity}]
Note that, for any knot $J$, the curve $d$ shown on the left hand side of Figure~\ref{fig:robust-with-ambiguity} is an unknotted derivative for $R^{k,J}$, and as a result, $R^{k,J}$ is ribbon. It is well-known that $\A(R^{k,J})$ is generated by $\{ [\alpha_k],[\beta_k]\}$. A direct Seifert matrix computation shows that $[\beta_k]=-k(1-t)[\alpha_k]$, and thus, $[\alpha_k]$ generates $\A(R^{k,J})$. Let $\delta_k(t) = kt-(k+1)$. Since $\Delta_{R^{k,J}}(t) = \delta_k(t^{})\delta_k(t^{-1})$, we see that $\A(R^{k,J})\cong\frac{\mathbb{Q}[t,t^{-1}]}{\langle\delta_k(t)\cdot\delta_k(t^{-1})\rangle}$. Notice also that $\delta_k(t)$ and $\delta_k(t^{-1})$ are prime and coprime if $k\neq 0,-1$. 

It remains to verify the last condition in the definition of a robust doubling operator. Since $R^{k,U}$ is ribbon, the set $\mathcal{FOS}(R^{k,U})$ contains zero, and thus, $\rho_0(J)\neq 0$ by assumption. From Remark~\ref{rem:3submodules}, since $\delta_k(t)$ and $\delta_k(t^{-1})$ are prime and coprime, we see that there are three isotropic submodules of $\A(R^{k,J})$, where the two nontrivial submodules have trivial intersection. We need to show that the three associated first-order signatures are either non-zero or correspond to ribbon disks for $R^{k,J}$.

Since $d$ is an unknotted derivative for $R^{k,J}$ for any $J$, we see that $\langle d \rangle$ corresponds to a ribbon disk for $R^{k,J}$ and thus, $\rho(M(R^{k,J}),\phi_{\langle d\rangle})=0$ by Remark~\ref{rmk:metalag}. As $R^{k,J} = R^{k,U}_{\eta}(J)$, there is a natural identification $i_*:\A(R^{k,J})\to \A(R^{k,U})$, sending $\langle d\rangle \subset \A(R^{k,J})$ to $\langle d\rangle\subset   \A(R^{k,U})$. In a slight abuse we have not included the dependence of $d$ on $J$ in our notation.  Since $\rho(M(R^{k,J}),\phi_{\langle d\rangle})=\rho(M(R^{k,U}),\phi_{\langle d\rangle})=0$ and $\rho_0(J)\neq 0$, Proposition~\ref{prop:rho-infection} implies that $[\eta]\in \langle d\rangle\subset \A(R^{k,U})$.  

From a Seifert matrix computation, it is easy to see that $[\eta]\neq 0$ in $\A(R^{k,U})$. Let $P$ be the nontrivial isotropic submodule of $\A(R^{k,J})$ which is not $\langle d\rangle$. As noted above, we know that $P\cap \langle d\rangle=\{0\}$, and thus, $[\eta]\notin i_*(P)$. From Proposition~\ref{prop:rho-infection}, we know that 
\[
\rho(M(R^{k,J}),\phi_P)= \rho(M(R^{k,U}),\phi_{i_*(P)})+\rho_0(J)\]
and 
\[
\rho(M(R^{k,J}),\phi_{\langle 0\rangle})= \rho(M(R^{k,U}),\phi_{\langle 0\rangle})+\rho_0(J).
\]
By hypothesis, $-\rho_0(J)\neq \rho(M(R^{k,U}),\phi_{i_*(P)})$, and thus, $\rho(M(R^{k,J}),\phi_P)\neq 0$. Similarly, since $-\rho_0(J)\neq \rho(M(R^{k,U}),\phi_{\langle 0\rangle})$, we see that $\rho(R^{k,J},\phi_{\langle 0\rangle})\neq 0$. \end{proof}

While the above robust doubling operators will be sufficient for our proof of Theorem~\ref{thm:cables-indep}, we will need a different collection of robust doubling operators for Theorem~\ref{thm:summand}. As an added benefit, these new operators will be completely explicit, without any ambiguity such as in the definition of $R^{k,J}_{\alpha_k}$ arising from the choice of the knot $J$. These new operators are denoted $Q_{\alpha_k}^{k}$, for $k\geq 1$, and are shown in Figure~\ref{fig:robust}. A direct computation shows that the Alexander polynomial of $Q^k$ is 
\[
\Delta_{Q^k}(t) = \left(kt-(k+1)\right)\left((k+1)t-k\right).
\] 
We now show that these doubling operators are robust for $k\geq 3$ (see also~\cite[Example~7.3]{Cochran-Harvey-Leidy:2011-1} and~\cite[Theorem~7.2.1]{Davis:2012-1}). Note that our examples are the mirror images of the examples in~\cite{Cochran-Harvey-Leidy:2011-1, Davis:2012-1}.

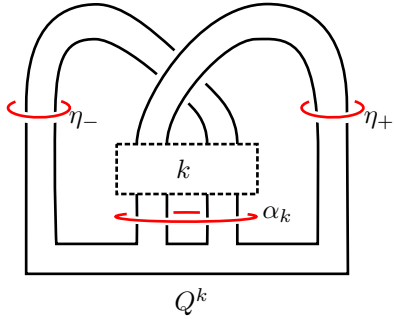
\begin{figure}[htb]
\labellist
\small\hair 2pt
\pinlabel $Q^k$ at 70 -10
\pinlabel $\alpha_k$ at 102 23
\pinlabel $k$ at 67 40
\pinlabel $\eta_+$ at 140 58
\pinlabel $\eta_-$ at 30 58

\endlabellist
\centering
\includegraphics{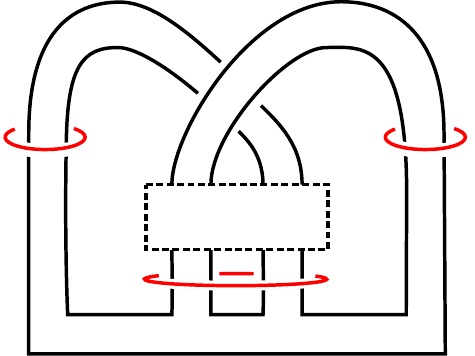}
\vspace*{5mm}
\caption{The  doubling operator $Q_{\alpha_k}^{k}$ (see Figure~\ref{fig:infection} for the case $k=1$). The dashed box containing $k$ indicates that the \emph{bands} passing through the box have $k$ full twists rather than all the strands. }\label{fig:robust}
\end{figure}

\begin{proposition}~\label{prop:robust} 
For $k\geq 3$, $Q_{\alpha_k}^{k}$ is a robust doubling operator.
\end{proposition}

\begin{proof} 
Let $\delta_k(t) = kt-(k+1)$. By a direct computation using the obvious genus one Seifert surface visible in Figure~\ref{fig:robust}, we see that 
\[
\mathcal{A}(Q^k)=\frac{\mathbb{Q}[t,t^{-1}]}{\langle\delta_k(t)\rangle}\oplus \frac{\mathbb{Q}[t,t^{-1}]}{\langle\delta_k(t^{-1})\rangle},
\]
where the generators are dual curves to the bands of this Seifert surface, namely the curves~$\eta_\pm$ in the figure. As before, for $k\notin\{-1,0\}$, $\delta_k(t)$ and $\delta_k(t^{-1})$ are prime and coprime. As a result, we see that $\A(Q^k)\cong\frac{\mathbb{Q}[t,t^{-1}]}{\langle\delta_k(t)\cdot\delta_k(t^{-1})\rangle}$ is generated by $[\alpha_k]=[\eta_+]+[\eta_-]$ (where as usual we are suppressing the orientations of the curves). 

It remains to verify the last condition in the definition of a robust doubling operator. We consider the three proper submodules $\langle 0 \rangle$ and $\langle \eta_\pm \rangle$, as noted in Remark~\ref{rem:3submodules}. Each of the submodules $\langle \eta_\pm \rangle$ corresponds to a ribbon disk $\Delta_\pm$ for $Q^{k}$ where $\Delta_+$ (resp.\ $\Delta_-$) is a disk obtained by cutting the band passing through $\eta_+$ (resp.\ $\eta_-$). We then only need to check that ~$\rho(M(Q^k),\phi_{\langle 0\rangle}) \neq 0$. In the earlier work of~\cite{Cochran-Harvey-Leidy:2011-1}, the  authors bypassed this computation by leaving an ambiguity in their doubling operators, consisting of a choice of whether or not to tie a trefoil in the left-hand band. This ambiguity was removed by the first author in his PhD thesis~\cite[Theorem~7.2.1]{Davis:2012-1}; we briefly sketch his argument now. Using an additivity result for first-order signatures \cite[Proposition~3.2.7]{Davis:2012-1}, he shows that there is a Lagrangian $P$ of $\A(Q^k\# Q^k)$ for which $\rho(M(Q^k\# Q^k), \phi_P)  = 2\rho(M(Q^k),\phi_{\langle 0\rangle})$. Then, he uses his result~\cite[Theorems~4.1.4 and~5.3.7]{Davis:2012-1} giving a bound on $\rho(M(Q^k\#Q^k),\phi_{P})$ in terms of the Cimasoni-Florens signature~\cite{Cimasoni-Florens:2008-1} of any derivative $d$ generating $P$. After finding a derivative generating~$P$, the proof reduces to a fairly explicit estimation of the Cimasoni-Florens signature of a link using a $C$-complex. \end{proof}

Given a doubling operator $R_\alpha$, let $(R_{p,1})_{\alpha '}$ denote the doubling operator obtained by taking the $(p,1)$ cable of $R_\alpha$ for some integer $p$, as shown in Figure~\ref{fig:robust-with-ambiguity}. More precisely, we consider the ribbon knot $R_{p,1}$ and the curve $\alpha '$ which is the image of $\alpha$ in $S^3\setminus R_{p,1}$. Observe that for any knot $K$, the $(p,1)$ cable of $R_\alpha(K)$ is isotopic to $(R_{p,1})_{\alpha'}(K)$. The following is the key technical result of this paper, which states that $(R_{p,1})_{\alpha'}$ is often robust when $R_\alpha$ is robust. 

\begin{theorem}~\label{prop:cablerobust} Let $R_\alpha$ be a robust doubling operator and let $\Delta_R(t) = \delta(t) \cdot \delta(t^{-1})$ be the Alexander polynomial of $R$. Let $p$ be a non-zero integer. If $\delta(t^p)$ is prime in $\Q[t,t^{-1}]$, then~$\left(R_{p,1}\right)_{\alpha'}$ is a robust doubling operator.
\end{theorem}

We postpone the proof until the end of the section. For now, we obtain the following corollary. 

\begin{corollary}~\label{cor:robust2} 
For $k\geq 1$, $p\ge 1$, and any knot $J$ with $-\rho_0(J)\notin \mathcal{FOS}(R^{k,U})$, the doubling operator $\left(R_{p,1}^{k,J} \right)_{\alpha'_k}$ is robust.

For $k\ge 3$ and $p\ge 1$, the doubling operator $\left(Q^{k}_{p,1}\right)_{\alpha_k'}$ is robust.
\end{corollary}

Before giving the proof, we recall the following results. The first is a generalization of Eisenstein's criterion for irreducibility of polynomials in $\Z[t]$, and the second is a result of the first author and Bullock concerning strong primality of polynomials in $\Z[t]$. 

\begin{proposition}[{\cite[Theorem~1.1]{Bonciocat}}]\label{prop:Bonciocat} Let $f(t)=a_0+a_1t+\cdots +a_dt^d\in \Z[t]$, $a_0a_d\neq 0$, and let $q_1$ and $q_2$ be distinct prime integers. Let $r^i_j \in\N\cup\{\infty\}$ be the maximum such that~$q_i^{r^i_j}$ divides $a_j$, where by convention $r^i_j = \infty$ if $a_j = 0$. 

Suppose for each $i=1,2$, 
\[
r^i_j\geq \frac{d-j}{d}r^i_0 +\frac{j}{d}r^i_d
\]
for all $j=1,\dots, d-1$, and where exactly one of $r^i_0$ and $r^i_d$ is non-zero; let the non-zero one of the two be denoted by $\alpha_i$. If $\gcd(\alpha_1,d)$ and $\gcd(\alpha_2,d)$ are relatively prime, then $f(t)$ is irreducible in $\Z[t]$. 

Consequently, if $f(t)$ is not a monomial, it is prime in~$\Q[t,t^{-1}]$. 
\end{proposition}

\begin{proposition}[{\cite[Corollary~4.4]{Bullock-Davis:2012-1}}]\label{prop:BDcriterion} Let $f(t)=a_0+a_1t+\cdots +a_dt^d\in \Z[t]$, where~$a_1$ and $a_0$ are relatively prime non-zero integers. If $f(t)$ is prime in $\Z[t]$ and $a_0 \neq \pm x^c$ for any integer $x$ and natural number $c > 1$, then $f(t)$ is strongly prime in $\Z[t]$. 

Consequently, if $f(t)$ is not a monomial, it is strongly prime in $\Q[t,t^{-1}]$. 
\end{proposition}
\noindent Note that in each of the above propositions, the last statement follows from Remark~\ref{rem:primality}.

\begin{proof}[Proof of Corollary \ref{cor:robust2}]Recall that
\[
\Delta_{R^{k,J}}(t)=\Delta_{Q^k}(t)=\delta_k(t)\delta_k(t^{-1}),
\]
where $\delta_k(t) = kt-(k+1)$ as before. We will now check that $\delta_k(t^p)$ is prime for all $p\geq 1$; this will complete the proof by Propositions~\ref{prop:robust-with-ambiguity} and~\ref{prop:robust} and Theorem~\ref{prop:cablerobust}. 

First let $k=8$. Then $\delta_k(t^p)=8t^p-9$. In Proposition~\ref{prop:Bonciocat}, let $q_1=2$ and $q_2=3$. Then for all $j=1,\dots, p-1$, $r^i_j=\infty$, $\alpha_1=r^1_p=3$ and $\alpha_2=r^2_0=2$. Observe that $\gcd(\alpha_1,p)$ is $1$ if $p\neq 3$ and is $3$ if $p=3$. Similarly, $\gcd(\alpha_2,p)=1$ if $p\neq 2$ and $2$ if $p=2$. In all cases, $\gcd(\gcd(\alpha_1,p),\gcd(\alpha_2,p))=1$ and so $\delta_k(t^p)$ is prime in $\Q[t,t^{-1}]$.  Thus we conclude that~$\delta_k(t)=8t^p-9$ is strongly prime in $\Q[t,t^{-1}]$ when $k=8$.  

Now let $k\neq 8$ and suppose for the sake of contradiction that there is some $p$ for which~$\delta_k(t^p)$ is not prime. We know from before that $\delta_k(t)$ is prime. As a result, $\delta_k(t^{-1})$ is prime, $\delta_k(t^{-p})$ is not, and $p\neq \pm 1$. In other words, $\delta_k(t)$ and $\delta_k(t^{-1})$ are both prime but not strongly prime. Then, by Proposition~\ref{prop:BDcriterion}, there exist positive integers $x, y$ and natural numbers $a,b >1$ such that $k+1 = x^a$ and $k = y^b$. In particular, $x^a - y^b = 1$. The solutions to this equation form the subject of the famous Catalan conjecture from 1844, which was proven in 2004 by Mih\u{a}ilescu in \cite{Mihailescu:2004-1} who showed that the only solution is~$x = 3, a = 2, y = 2, b = 3$. This implies that $k = 8$ which is a contradiction. In conclusion, $\delta_k(t)$ is strongly prime in $\Q[t,t^{-1}]$ for all~$k$, which completes the proof. \end{proof}

\noindent We end the section with the promised proof of Theorem~\ref{prop:cablerobust}.

\begin{proof}[Proof of Theorem~\ref{prop:cablerobust}]
We have already seen that $(R_{p,1})_{\alpha'}$ is a doubling operator and so we only need to show that it is robust. The first half of Definition \ref{def:robust} will follow from the fact that the Alexander module of $R_{p,1}$ is a ``tensored up'' version of the Alexander module of~$R$. This is a well-known fact; we give a proof since we will need the details in the second half of the proof. Our strategy is similar to \cite{Seifert:1950-1, Livingston-Melvin:1985-1}. 

Let $\Delta$ be a ribbon disk for $R$. Let $N(\Delta)$ be a tubular neighborhood of $\Delta$ restricting to~$N(R)$ a tubular neighborhood of $R$. Construct $R_{p,1}$ within $N(R)$ and construct a ribbon disk $\Delta_{p,1}$ for $R_{p,1}$ by taking $p$ parallel copies of $\Delta$ within $N(\Delta)$ and banding them together while introducing only index one critical points. Let $N(R_{p,1})$ and $N(\Delta_{p,1})$ be tubular neighborhoods of $R_{p,1}$ and $\Delta_{p,1}$ contained within $N(R)$ and $N(\Delta)$ respectively. Let $E(R)$, $E(R_{p,1})$, $E(\Delta)$, and $E(\Delta_{p,1})$ be the exteriors of the tubular neighborhoods of $R$, $R_{p,1}$, $\Delta$, and $\Delta_{p,1}$ constructed above. 

\begin{figure}[htb]
\centering
\includegraphics{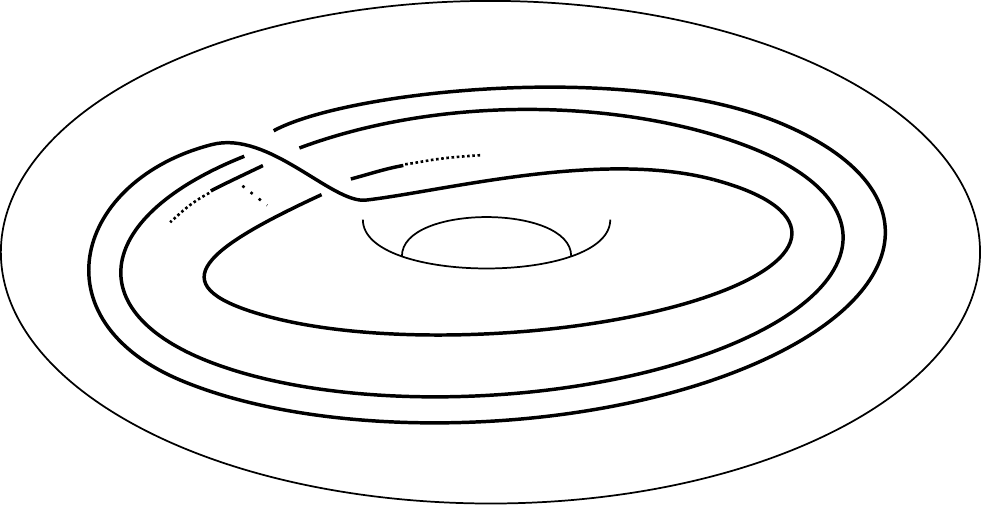}
\caption{The complement of the knot shown above is the space $X$. The pictured knot wraps $p$ times around the longitude of the solid torus and only once around the meridian.}\label{fig:cableexterior}
\end{figure}
Let $X = \overline{N(R)}\setminus N(R_{p,1})$ be the exterior of the knot in a solid torus depicted in Figure~\ref{fig:cableexterior}. Let $Y$ be the space $\overline{N(\Delta)}\setminus N(\Delta_{p,1})$. From our construction, we see that 
\[E(R_{p,1})=E(R)\cup X
\]
and 
\[E(\Delta_{p,1})=E(\Delta) \cup Y.
\]
Additionally, note that $H_1(E(R_{p,1}))\cong H_1(E(\Delta_{p,1}))\cong \Z$ is generated by a meridian of~$R_{p,1}$ while $H_1(E(R))\cong H_1(E(\Delta))\cong \Z$ is generated by a meridian of $R$. Moreover, within~$E(R_{p,1})$, the meridian of $R$ is homologous to $p$ times the meridian of $R_{p,1}$. 

Let $\pi_1(E(\Delta_{p,1}))\to \Z=\langle s\rangle$ and $\pi_1(E(\Delta))\to \Z=\langle t\rangle$ be the abelianization maps. The~$\Q[s,s^{-1}]$-modules $H_1(E(R_{p,1}); \Q[s,s^{-1}])$ and  $H_1(E(\Delta_{p,1}); \Q[s,s^{-1}])$ are by definition the rational Alexander modules $\A(R_{p,1})$ and $\A(\Delta_{p,1})$ of $R_{p,1}$ and $\Delta_{p,1}$ respectively.   Similarly, the $\Q[t,t^{-1}]$-modules $H_1(E(R); \Q[t,t^{-1}])$ and  $H_1(E(\Delta); \Q[t,t^{-1}])$ are the rational Alexander modules $\A(R)$ and $\A(\Delta)$ of $R$ and $\Delta$ respectively. Define a $\Q[t,t^{-1}]$-module structure on $\Q[s,s^{-1}]$ by letting $t$ act by multiplication by $s^p$. Then $\Q[s,s^{-1}]$ is free as a module over $\Q[t,t^{-1}]$ with basis $1,s,\dots, s^{p-1}$. As a result, $\Q[s,s^{-1}]$ is flat as a module over $\Q[t,t^{-1}]$ and we see that 
\[
H_1(E(R);\Q[s,s^{-1}]) \cong \A(R)\otimes_{\Q[t,t^{-1}]} \Q[s,s^{-1}]\]
and 
\[H_1(E(\Delta);\Q[s,s^{-1}])\cong \A(\Delta)\otimes_{\Q[t,t^{-1}]} \Q[s,s^{-1}].
\]
  
Let $T$ be the torus $E(R)\cap X$ and $S$ be the solid torus $E(\Delta)\cap Y$. Consider the following diagram of homology groups with coefficients in $\Q[s,s^{-1}]$, where the rows come from the Mayer-Vietoris exact sequence, and the vertical arrows are induced by inclusion.
  \begin{equation}\label{eqn:MVScable}
  \begin{diagram}
  \node{H_d(T)}\arrow{e,t}{i_*\oplus j_*}\arrow{s} \node{H_d(E(R))\oplus H_d(X)} \arrow{e}\arrow{s} \node{H_d(E(R_{p,1}))}\arrow{e,t}{}\arrow{s}\node{H_{d-1}(T)}\arrow{s}\\
  \node{H_d(S)}\arrow{e,t}{k_*\oplus \ell_*} \node{H_d(E(\Delta))\oplus H_d(Y)} \arrow{e}\node{H_d(E(\Delta_{p,1}))}\arrow{e,t}{}\node{H_{d-1}(S)}
  \end{diagram}
  \end{equation}
Above, the maps $j_*$ and $\ell_*$ are induced by the inclusions $T\hookrightarrow X$ and $S\hookrightarrow Y$ respectively.  These maps are independent of $R$ and $\Delta$ and thus, we may use any choice of~$R$ and~$\Delta$ to investigate them. Consider the case when $R$ is the unknot and $\Delta$ is obtained by taking a disk bounded by $R$ in $S^3$ and pushing the interior into $B^4$. In this case,~$R_{p,1}$ is also unknotted and~$\Delta_{p,1}$ is a pushed in copy of a disk in $S^3$. We see that for all~$d>0$, $H_d(E(R))\cong H_d(E(\Delta))\cong H_d(E(R_{p,1}))\cong H_d(E(\Delta_{p,1}))=0$ (recall that we are using~$\Q[s,s^{-1}]$ coefficients). Thus, the diagram above shows that $j_*$ and $\ell_*$ are isomorphisms when $d\neq 0$ and are monomorphisms for $d=0$. Returning to the general case, we see immediately that 
\[
H_d(E(R);\Q[s,s^{-1}]) \rightarrow H_d(E(R_{p,1});\Q[s,s^{-1}])
\]
and 
\[H_d(E(\Delta);\Q[s,s^{-1}]) \to H_d(E(\Delta_{p,1});\Q[s,s^{-1}])
\]
are isomorphisms for all $d > 0$. In particular, we have now established that the inclusion map $E(R)\to E(R_{p,1})$ induces an isomorphism
\[
\A(R_{p,1}) \cong \A(R)\otimes \Q[s,s^{-1}].
\]
In particular, we see that
\[
\A(R_{p,1}) \cong \dfrac{\Q[t,t^{-1}]}{\langle\delta(t)\cdot\delta(t^{-1})\rangle}\otimes \Q[s,s^{-1}] \cong\dfrac{\Q[s,s^{-1}]}{\langle\delta(s^p)\cdot\delta(s^{-p})\rangle}
\]
is generated by $\alpha'$, the image of $\alpha$ in $\A(R_{p,1})$. By hypothesis $\delta(s^p)$ is prime in~$\Q[s,s^{-1}]$ and thus, $\left(R_{p,1}\right)_{\alpha'}$ satisfies the first condition for being a robust doubling operator.

It remains to verify the last condition in the definition of a robust doubling operator. That is, we show that for every isotropic submodule $P$ of $\A(R_{p,1})$, either $P$ corresponds to a ribbon disk for~$R_{p,1}$ or~$\rho(M(R_{p,1}),\phi_P)\neq 0$.  From Remark~\ref{rem:3submodules}, we know that there are only three proper submodules of $\A(R_{p,1})$, namely $\langle 0\rangle$, $\langle \delta(s^p)\cdot [\alpha']\rangle$, and $\langle\delta(s^{-p})\cdot [\alpha']\rangle$. Observe that each of these is realized as $Q\otimes \Q[s,s^{-1}]$ for some submodule $Q$ of $\A(R)$. More precisely, 
\begin{align*}
\langle 0\rangle &= \langle 0\rangle \otimes_{\Q[t,t^{-1}]} \Q[s,s^{-1}]\\
\langle \delta(s^p)\cdot [\alpha']\rangle &= \langle \delta(t)\cdot [\alpha]\rangle \otimes_{\Q[t,t^{-1}]} \Q[s,s^{-1}]\\
\langle \delta(s^{-p})\cdot [\alpha']\rangle &= \langle \delta(t^{-1})\cdot [\alpha]\rangle \otimes_{\Q[t,t^{-1}]} \Q[s,s^{-1}]
\end{align*}
Let $Q$ be an isotropic submodule of $\A(R)$ corresponding to a ribbon disk $\Delta$ for $R$, that is,~$Q = \ker(\A(R)\to \A(\Delta))$. We claim that $Q\otimes\Q[s,s^{-1}]$ corresponds to $\Delta_{p,1}$.  Indeed, using the identification of $\A(R_{p,1})$ with $\A(R)\otimes \Q[s,s^{-1}]$ and applying~\eqref{eqn:MVScable}, we observe that 
\begin{align*}
\ker(\A(R_{p,1})\to \A(\Delta_{p,1})) 
&= \ker(\A(R)\otimes \Q[s,s^{-1}]\to \A(\Delta)\otimes\Q[s,s^{-1}])\\
&= \ker(\A(R)\to \A(\Delta))\otimes\Q[s,s^{-1}]\\
&= Q\otimes \Q[s,s^{-1}].
\end{align*}
where the second equality uses the fact that $\Q[s,s^{-1}]$ is flat. Thus,  $Q\otimes \Q[s,s^{-1}]$ corresponds to a ribbon disk for $R_{p,1}$ whenever $Q$ corresponds to a ribbon disk for $R$.  

Let $Q\otimes \Q[s,s^{-1}]$ be an isotropic (and thus, proper) submodule of $\A(R_{p,1})$. By Remark~\ref{rem:3submodules}, we see that $Q$ is an isotropic submodule of $\A(R)$.   Suppose that $Q\otimes \Q[s,s^{-1}]$ does not correspond to a ribbon disk for $R_{p,1}$. As a result, we see that $Q$ does not correspond to a ribbon disk for $R$. The proof will be complete if we can show that the corresponding first-order signature, $\rho(M(R_{p,1}), \phi_{Q\otimes \Q[s,s^{-1}]})$ is not zero. Since $R_\alpha$ is a robust doubling operator, $\rho(M(R), \phi_Q)\neq 0$. We will establish that $\rho(M(R_{p,1}), \phi_{Q\otimes \Q[s,s^{-1}]})\neq 0$ by showing that 
\begin{equation}
\label{eqn:rho and cable}\rho(M(R_{p,1}), \phi_{Q\otimes \Q[s,s^{-1}]}) = \rho(M(R), \phi_Q).
\end{equation}

Construct a cobordism from $M(R)$ to $M(R_{p,1})$ as follows. Start with $M(R)\times [0,1]$ and attach a $1$-handle to $M(R)\times \{1\}$. The top boundary component is described by the surgery diagram given in Figure~\ref{fig:kirby1}(a). Attach a $0$-framed $2$-handle to the top boundary as shown in Figure~\ref{fig:kirby1}(b). The new boundary is diffeomorphic to $M(R_{p,1})$, as demonstrated in Figure~\ref{fig:kirby2}, via a handleslide followed by a slam dunk move. Let $V$ denote the resulting cobordism with $\partial V= -M(R)\sqcup M(R_{p,1})$. From our construction, it is clear that the meridian of $R$ in $M(R)$ is homologous in $V$ to $p$ times the meridian of $R_{p,1}$ in $M(R_{p,1})$. 
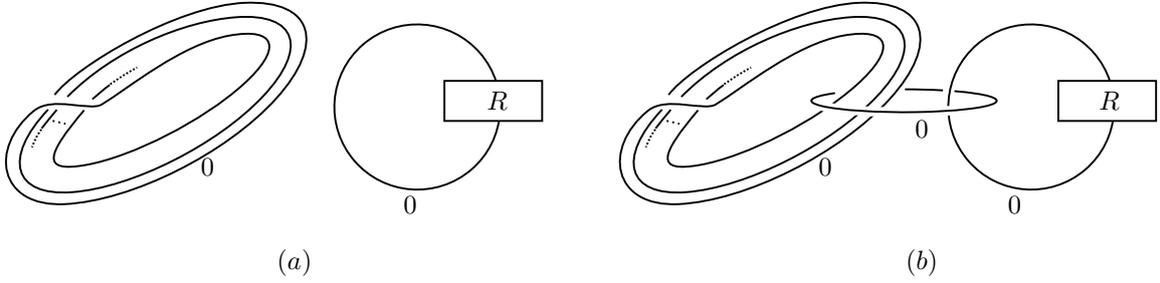
\begin{figure}[htb]
\labellist
\small\hair 2pt
\pinlabel $0$ at 105 20
\pinlabel $0$ at 210 0
\pinlabel $R$ at 255 55

\pinlabel $0$ at 425 20
\pinlabel $0$ at 523 0
\pinlabel $R$ at 572 55
\pinlabel $0$ at 475 40

\pinlabel $(a)$ at 150 -30
\pinlabel $(b)$ at 475 -30
\endlabellist
\centering
\includegraphics[width=\textwidth]{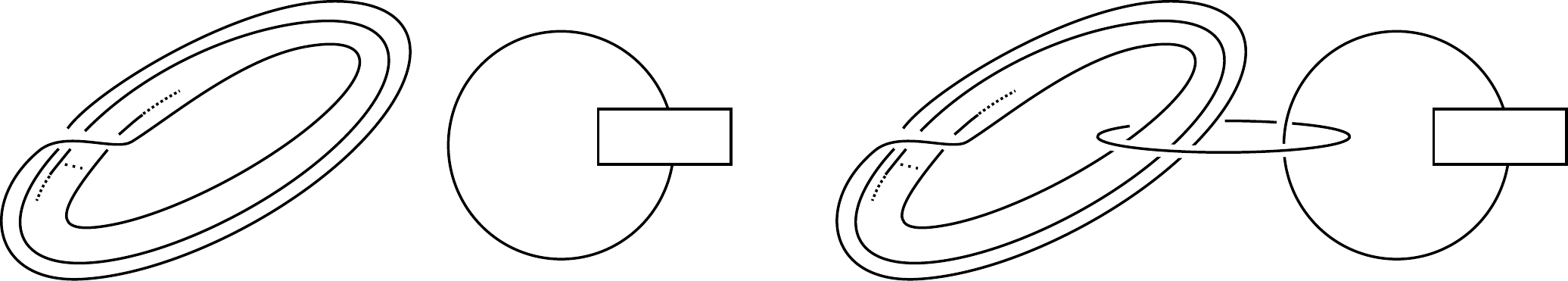}
\vspace*{5mm}
\caption{(a) The top boundary of the $4$-manifold obtained by attaching a $1$-handle to $M(R)\times [0,1]$. (b) Attach a $0$-framed $2$-handle to the manifold depicted in (a). In both figures, the leftmost knot is the same as in Figure~\ref{fig:cableexterior}. }\label{fig:kirby1}
\end{figure}

\begin{figure}[htb]
\labellist
\small\hair 2pt
\pinlabel $0$ at 105 50
\pinlabel $0$ at 210 35
\pinlabel $R$ at 255 90

\pinlabel $0$ at 377 70
\pinlabel $0$ at 370 0
\pinlabel $R$ at 415 80
\pinlabel $0$ at 400 105

\pinlabel $0$ at 520 0
\pinlabel $R$ at 565 80

\pinlabel $(a)$ at 150 -30
\pinlabel $(b)$ at 350 -30
\pinlabel $(c)$ at 500 -30
\endlabellist
\centering
\includegraphics[width=\textwidth]{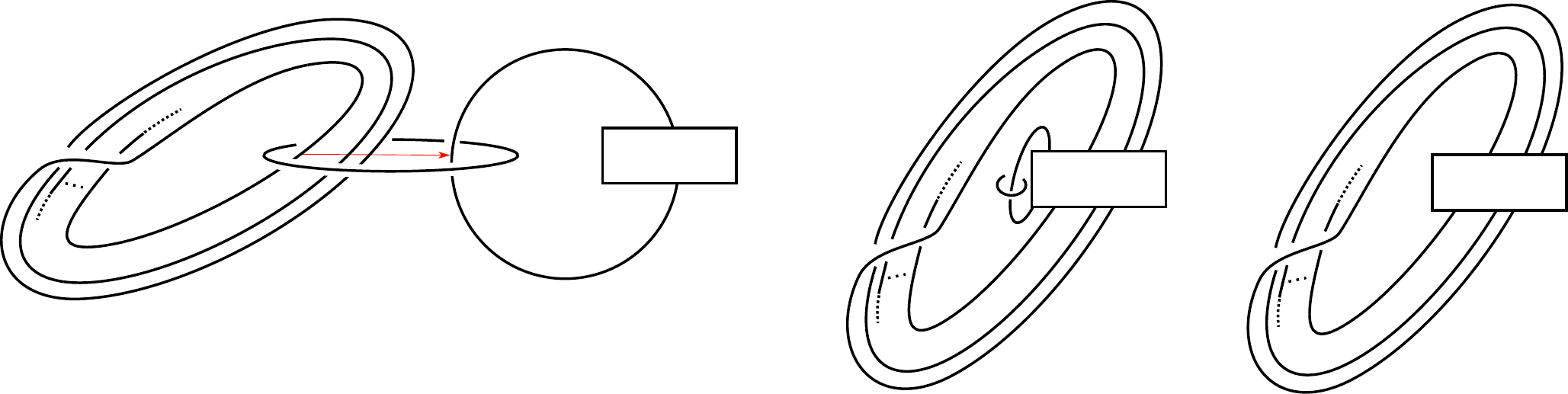}
\vspace*{5mm}
\caption{In (a), sliding the leftmost $2$-handle over the rightmost $2$-handle repeatedly, as indicated, yields the $3$-manifold shown in (b). Now perform a slam dunk move to obtain the manifold in (c). Observe that this is the manifold $M(R_{p,1})$ as desired.}\label{fig:kirby2}
\end{figure}
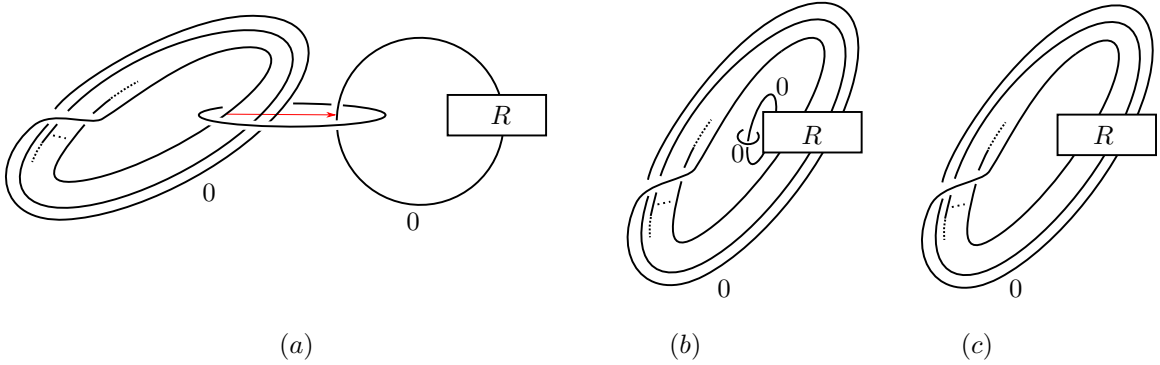

Note that the handles added to $M(R)$ to produce $V$ live within a neighborhood of $R$. That is, there is a copy of $E(R)\times[0,1]$ embedded in $V$ so that $E(R)\times \{0\}$ is sent to~$E(R)\subset M(R)$ and $E(R)\times\{1\}$ is mapped to the copy of $E(R)\subset E(R_{p,1})\subset M(R_{p,1})$ corresponding to $R_{p,1}$ being a satellite of $R$.   

Turning the handle structure of $V$ upside down, we see that $V$ can be built from~$M(R_{p,1})\times[0,1]$, by adding a $0$-framed $2$-handle along the curve $\ell$ of Figure~\ref{fig:kirby3}(b), performing handleslides over this new curve to produce a $0$-framed unknot, and then adding a $3$-handle to cancel this unknotted surgery curve; see Figure~\ref{fig:kirby3}. It follows immediately that $H_1(V)\cong \Z$ is generated by the meridian of $R_{p,1}$. Moreover, observe that $\ell$ bounds a surface in $M(R_{p,1})$ (given by a Seifert surface for $R$) which is disjoint from a Seifert surface for $R_{p,1}$. Thus, $\ell$ is a double commutator in $\pi_1(M(R_{p,1}))$, and as a result, it lifts to a nullhomologous curve in the infinite cyclic cover of $M(R_{p,1})$. In other words, $[\ell]=0$ in $H_1(M(R_{p,1});\Q[s,s^{-1}])$. This implies that the inclusion-induced map $H_1(M(R_{p,1});\Q[s,s^{-1}])\to H_1(V; \Q[s,s^{-1}])$ is an isomorphism.
\begin{figure}[htb]
\labellist
\small\hair 2pt
\pinlabel $0$ at 55 150
\pinlabel $R$ at 95 225
\pinlabel $(a)$ at 50 120

\pinlabel $0$ at 220 150
\pinlabel $R$ at 263 225
\pinlabel $\ell$ at 227 205
\pinlabel $0$ at 228 235
\pinlabel $(b)$ at 220 120

\pinlabel $0$ at 385 150
\pinlabel $R$ at 528 208
\pinlabel $0$ at 528 235
\pinlabel $(c)$ at 425 120

\pinlabel $0$ at 150 0
\pinlabel $0$ at 252 0
\pinlabel $R$ at 270 47
\pinlabel $(d)$ at 175 -30

\pinlabel $0$ at 445 0
\pinlabel $R$ at 457 47
\pinlabel $(e)$ at 415 -30
\endlabellist
\centering
\includegraphics[width=\textwidth]{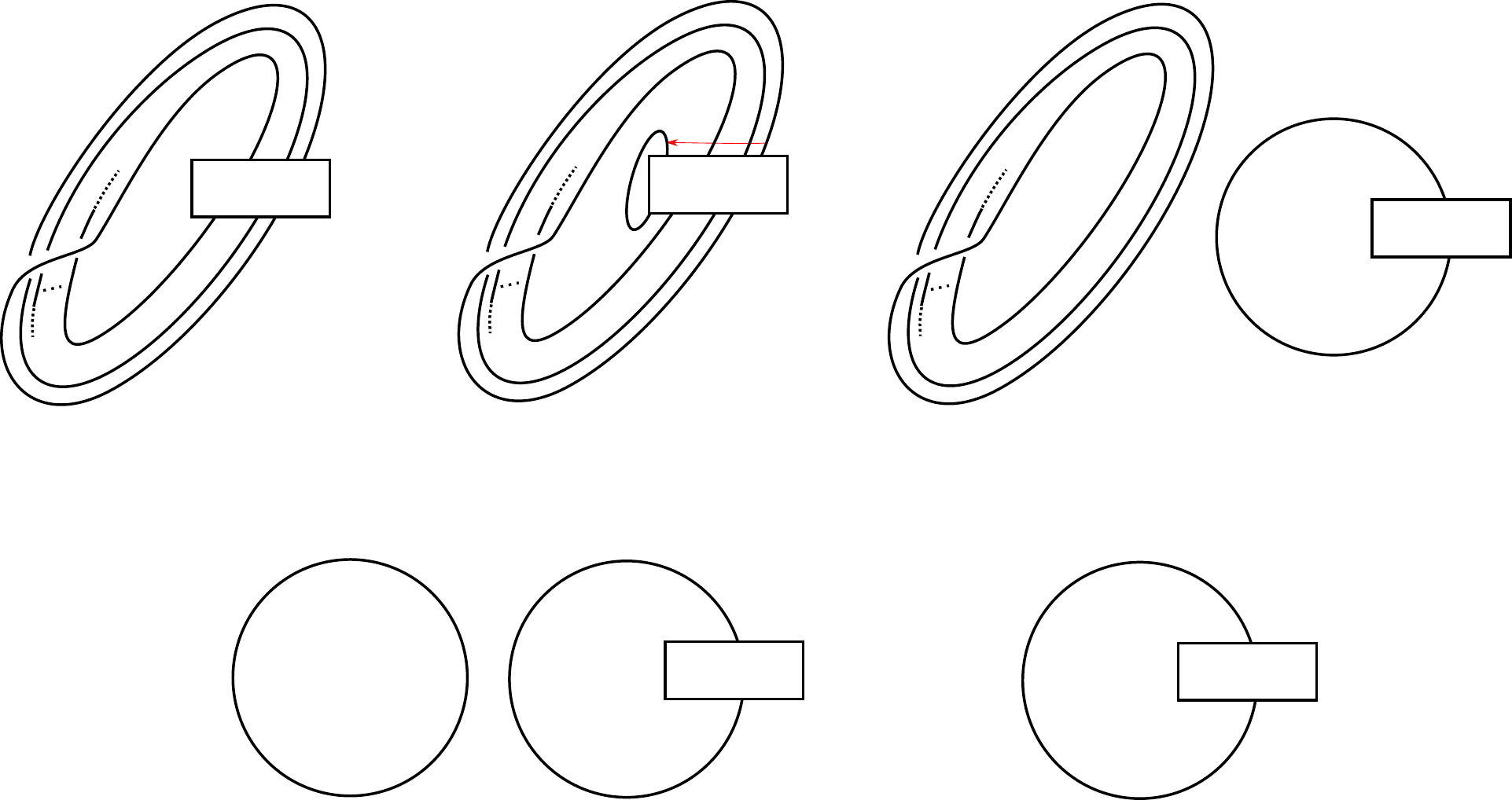}
\vspace*{5mm}
\caption{(a) The $3$-manifold $M(R_{p,1})$. (b) Attach a $0$-framed $2$-handle along the curve $\ell$ then repeatedly slide the other $2$-handle over it. (c) The result of the handle slides in (b). (d) The result of an isotopy on the diagram in (c). Attach a $3$-handle to cancel the $0$-framed unknot on the left. (e) The result of canceling the $2$-handle on the left of (d) and the attached $3$-handle.}\label{fig:kirby3}
\end{figure}
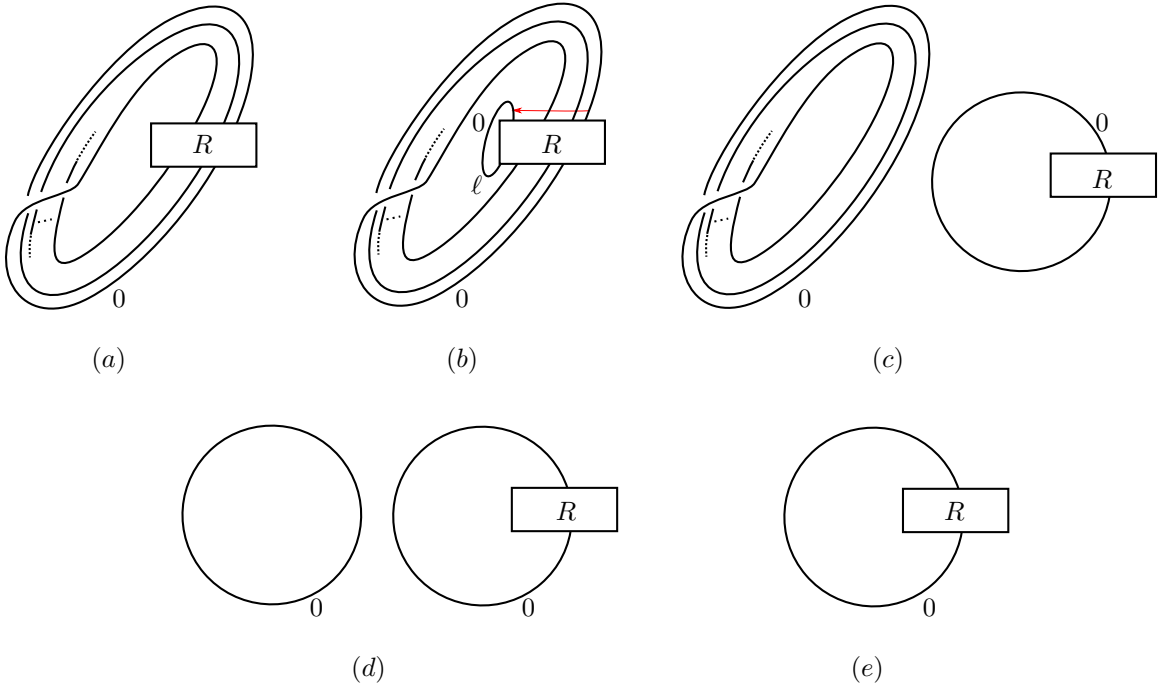

Since $\Q[s,s^{-1}]$ is flat as a $\Q[t,t^{-1}]$ module, and $H_1(E(R);\Q[t,t^{-1}])\rightarrow H_1(M(R);\Q[t,t^{-1}])$ is an isomorphism, the map $H_1(E(R);\Q[s,s^{-1}])\rightarrow H_1(M(R);\Q[s,s^{-1}])$ is also an isomorphism. Moreover, recall from earlier that $H_1(E(R);\Q[s,s^{-1}]) \rightarrow H_1(E(R_{p,1});\Q[s,s^{-1}])$ is an isomorphism. Combining our observations from the last few paragraphs, we obtain a sequence of inclusion-induced isomorphisms with $\Q[s,s^{-1}]$-coefficients:
\[
H_1(M(R))\cong H_1(E(R)) \cong H_1(E(R_{p,1}))\cong H_1(M(R_{p,1})) \cong H_1(V).
\]
Moreover, let $Q\subseteq \A(R)$ be any submodule and let (via an abuse of notation) $Q'$ denote the image of $Q\otimes \Q[s,s^{-1}]$ in the isomorphic modules $\A(V) := H_1(V;\Q[s,s^{-1}])$ and $\A(R_{p,1}):=H_1(E(R_{p,1};\Q[s,s^{-1}]))$. 
Define $\overline{\phi_{Q'}}\colon \pi_1(V) \to \Z\ltimes \frac{\A(V)}{Q'}$ analogously to the definition prior to Remark~\ref{rem:alex-module-defn}, that is,
\[
\overline{\phi_{Q'}}\colon \pi_1(V)\to \pi_1(V)/\pi_1(V)^{(2)}\cong \Z\ltimes \pi_1(V)^{(1)}/\pi_1(V)^{(2)}\to \Z\ltimes \A(V)\to  \Z\ltimes \A(V)/Q'.
\]
Here we are using the fact that $H_1(V)\cong \Z$. 
Then the following diagram commutes.
$$\begin{diagram}
\node{\pi_1(M(R))}\arrow{e,t}{\phi_Q}\arrow{s}\node{\langle t\rangle \ltimes \frac{\A(R)}{Q}}\arrow{s,J}
\\
\node{\pi_1(V)}\arrow{e,t}{\overline{\phi_{Q'}}}\node{\langle s\rangle \ltimes \frac{\A(V)}{Q'}}
\\
\node{\pi_1(M(R_{p,1}))}\arrow{n}\arrow{e,t}{\phi_{Q'}}\node{\langle s\rangle \ltimes \frac{\A(R_{p,1})}{Q'}}\arrow{n,l}{\cong}
\end{diagram}$$
To see that the map on the top right is injective, consider the following commutative diagram, whose rows are exact.
\[
\begin{diagram}
\node{0} \arrow{e}\node{\A(R)/Q} \arrow{e}\arrow{s,t}{x\mapsto x\otimes 1} \node{\langle t\rangle \ltimes \A(R)/Q} \arrow{e}\arrow{s} \node{\langle t \rangle} \arrow{e}\arrow{s,t}{t\mapsto  s^p} \node{0}\\
\node{0} \arrow{e}\node{\A(V)/Q'} \arrow{e} \node{\langle s\rangle \ltimes \A(V)/Q'} \arrow{e} \node{\langle s \rangle} \arrow{e} \node{0}
\end{diagram}
\]
Since the adjoining maps are injective, the middle map is injective by the five lemma. 
By the properties of von Neumann $\rho$-invariants listed in Section~\ref{sec:L2-signature-invariants}, we see that 
\[
-\rho(M(R), \phi_Q) + \rho(M(R_{p,1}), \phi_{Q'}) = \sigma^{(2)}(V, \phi_{Q'}) - \sigma(V)
\]
Recall that $V$ is built from $M(R)$ by adding a $1$-handle and then a $2$-handle along a homologically nontrivial curve. Thus, $H_2(M(R))\to H_2(V)$ is an epimorphism and so~$\sigma(V) = 0$. 

We have now reduced the proof to showing that $\sigma^{(2)}(V,\phi_{Q'})=0$. As a preliminary, we first show that $\Gamma$ is \emph{PTFA (poly-torsion-free-abelian)}. Recall that a group $G$ is said to be PTFA if it admits a descending series of normal subgroups 
\[
0=G_n\lhd G_{n-1}\lhd \dots \lhd G_0=G
\]
such that $G_{k}/G_{k+1}$ is torsion-free and abelian~\cite[Definition 2.5]{Cochran:2002-1}. In our case,  $\Gamma = \langle s\rangle \ltimes \frac{\A(V)}{Q'}$, and we have the series $0\lhd G_1 \lhd \Gamma$ where $G_1=\frac{\A(V)}{Q'}\le \Gamma$. Indeed, $G_1 = \frac{\A(V)}{Q'}$ is a vector space over $\Q$ and so is torsion-free and abelian, while $\Gamma/G_1 \cong \langle s\rangle\cong \Z$.

There is a bound on the $L^2$-signature of a $4$-manifold in terms of the rank of a homology group with twisted coefficients. 
We briefly summarize where this bound comes from (see also \cite[Section 2.3]{Davis:2012-1}). For any $4$-manifold $X$ and group $\Gamma$, the $L^2$-signature is defined in terms of the $L^2$-homology $H_2^{(2)}(X;\ell^2(\Gamma))$. Indeed, by \cite[Corollary 1.10]{Lueck-Schick:2003-1}, $\sigma^{(2)}(X,\psi)$ is the difference in $L^2$-dimensions of the maximal subspace on which the intersection form is positive definite and the maximal subspace where the intersection form is negative definite. (For a summary of $L^2$-homology and $L^2$-dimension, see \cite{Lueck:2002-1}.)  The monotonicity of the $L^2$-dimension \cite[Theorem 1.12 (2)]{Lueck:2002-1} immediately implies that
$$
\lvert \sigma^{(2)}(X,\psi)\rvert \leq \dim^{(2)}\left({H^{(2)}_2(X;\ell^2(\Gamma))} \right).
$$
According to \cite[Proposition 3.2]{Cochran:2002-1}, since $\Gamma$ is PTFA, $\Q[\Gamma]$ is an Ore domain and so embeds into its skew field of fractions $\K(\Gamma)$. In this case, by~\cite[Lemma~2.4]{Cha:L2} the $L^2$-dimension agrees with rank over $\K(\Gamma)$, and thus, 
\begin{equation}\label{eq:L2-bound}
\lvert \sigma^{(2)}(X,\psi)\rvert \leq \rk_{\K(\Gamma)}\left({H_2(X;\K(\Gamma))} \right)
\end{equation}

Finally, we complete the proof by computing $H_2(X;\K(\Gamma))$. Note that $\overline \phi_{Q'}$ is nontrivial as it sends a meridian of $R_{p,1}$ to $s\in \langle s\rangle \ltimes \frac{\A(V)}{Q'}$. By~\cite[Proposition~3.7]{Cochran:2002-1}, since $V$ is a connected CW complex, it follows that $H_0(V;\K(\Gamma))=0$. Then, by~\cite[Proposition~3.10 and Remark~3.6(1)]{Cochran:2002-1}, since $\pi_1(V)$ is finitely generated, $\phi_{Q'}$ is nontrivial, and $\beta_1(V)=1$, it follows that $H_1(V;\K(\Gamma))=0$. Notice that $\overline\phi_{Q'}$ sends the meridian of $R_{p,1}$ to $s$ and the meridian of $R$ to $s^p$, both of which are nontrivial in $\Gamma$. Thus, $\overline\phi_{Q'}$ is nontrivial even when restricted to any boundary component of $V$.  
The same argument as we used on $V$ implies that $H_1(M(R);\K(\Gamma))=H_1(M(R_{p,1});\K(\Gamma))=0$. Since $V$ is a $4$-manifold with boundary,~$H_4(V; \K(\Gamma)) = 0$. Recall from our construction of $V$ that $(V,M(R))$ has no relative $3$-handles. Thus, $H_3(M(R); \K(\Gamma)) \to H_3(V; \K(\Gamma))$ is onto. However, 
\[
H_3(M(R);\K(\Gamma))\cong H^1(M(R);\K(\Gamma))\cong \Hom(H_1(M(R);\K(\Gamma)),\K(\Gamma)) = \Hom(0,\K(\Gamma))=0.\]
Thus, $H_3(V;\K(\Gamma))=0$. Finally, since $\chi(V) = 0$, \cite[Page 357, Fact 3]{Cochran:2002-1} implies that~$H_2(V; \K(\Gamma))=0$.  By the bound~\eqref{eq:L2-bound}, we finally conclude that $\sigma^{(2)}(V, \ \phi_{Q'})=0$, which completes the proof. \end{proof}

\section{Proof of Theorem~\ref{thm:cables-indep}}\label{sec:proofofA}

Cochran-Harvey-Leidy give a very general construction of knots in $\F_n$ which are linearly independent modulo $\F_{n.5}$ \cite[Theorem~7.7]{Cochran-Harvey-Leidy:2011-1}. Moreover, as we mention in Remark~\ref{rem:rho-obstructions}, in \cite{Cochran-Harvey-Horn:2013-1} Cochran-Harvey-Horn demonstrate that the tools used in \cite{Cochran-Harvey-Leidy:2011-1} also obstruct a knot being in $\B_{n+1}$, so that the following is a consequence.  

\begin{theorem}~\label{cor:lin-ind} Let $\mathcal{R}^i = \left(R^{i,1}_{\alpha_{i,1}},\ldots,R^{i,n}_{\alpha_{i,n}}\right)$, for $i \geq 1$, be sequences of robust doubling operators and $\mathcal{Q}^i=\left(\Delta_{R^{i,n}}(t), \ldots,\Delta_{R^{i,1}}(t)\right)$ be the sequences of Alexander polynomials of~$\mathcal{R}^i$.
Let $\{K^{m,i}\}_{m,i \geq 1}$ be an infinite set of knots.
Suppose that 
\begin{enumerate}
\item \label{item:coprime}$\mathcal{Q}^i$ is strongly coprime to $\mathcal{Q}^{i'}$ for all $i\neq i'$, and
\item \label{item:rho} for each fixed $i$, no nontrivial  rational linear combination of elements of the set $\{\rho_0(K^{m,i})\}_{m\geq 1}$ produces an element in the rational span of $\mathcal{FOS}(R^{i,1})$.
\end{enumerate}
Then $\{R^{i,n}_{\alpha_{i,n}}\circ \cdots \circ R^{i,1}_{\alpha_{i,1}}(K^{m,i})\}_{m,i\geq 1}$ is linearly independent in $\C/(\mathcal{F}_{n.5}+\B_{n+1})$.
\end{theorem}

\begin{remark}
The reader will notice that item (\ref{item:rho}) above differs slightly from the wording in \cite{Cochran-Harvey-Leidy:2011-1}.  In \cite[Theorem~7.7]{Cochran-Harvey-Leidy:2011-1} the spans of the two sets are merely assumed to have trivial intersection, but this may produce examples for which the theorem fails.  For example, if every $K^{m,i}$ is the unknot, then each $R^{i,n}_{\alpha_{i,n}}\circ \cdots \circ R^{i,1}_{\alpha_{i,1}}(K^{m,i})$ will be slice, while the span of $\{\rho_0(K^{m,i})\}_{m\geq 1} = \{0\}$ is the trivial subspace of $\mathbb R$, which has trivial intersection with every subspace.  
\end{remark}

\noindent Combining the above with Theorem~\ref{prop:cablerobust}, we obtain the following corollary. The reader will see that we have lost generality by taking $\mathcal{R}^i = \mathcal{R}$ to be the same for all $i$.  The only reason to do so is to avoid an excess of notation.  

\begin{corollary} Let $\mathcal{R}= \left(R^{1}_{\alpha_{1}},\ldots,R^{n}_{\alpha_{n}}\right)$ be a sequence of robust doubling operators and let~$\Delta_{R^{n}}(t)$ be the Alexander polynomial of $R^n$.
  Let $\{K^{m}\}_{m \geq 1}$ be an infinite family of knots. 
Suppose that
\begin{enumerate}
\item for all $p\in \N$, $\Delta_{R^n}(t^p) = \delta(t^p)\cdot \delta(t^{-p})$ where $\delta(t^p)$ is prime in $\Q[t,t^{-1}]$, and
\item no nontrivial  rational linear combination of elements of $\{\rho_0(K^{m})\}_{m\geq 1}$ produces an element in the rational span of $\mathcal{FOS}(R^{1})$.
\end{enumerate}
Then $\{J^m_{p,1}\}_{m,p\geq 1}$ is linearly independent in $\C/(\mathcal{F}_{n.5}+\B_{n+1})$, where $J^m$ denotes the knot~$R^n_{\alpha_n}\circ \cdots \circ R^1_{\alpha_1}(K^m)$.
\end{corollary}

\begin{proof} 
Let $\mathcal{R}^p = \left(R^{1}_{\alpha_{1}},\ldots,R^{n-1}_{\alpha_{n-1}},(R^{n}_{p,1})_{\alpha_n'}\right)$ for each $p\geq 1$, where as before, $(R^{n}_{p,1})_{\alpha_n'}$ denotes the doubling operator obtained as the $(p,1)$ cable of $R^{n}$.  By Theorem~\ref{prop:cablerobust}, $(R^{n}_{p,1})_{\alpha_n'}$ is robust. Let $\mathcal{Q}^p= \left(\Delta_{R^n}(t^p),\Delta_{R^{n-1}}(t), \ldots,\Delta_{R^1}(t)\right)$ be the sequence of Alexander polynomials of $\mathcal{R}^p$. Consider any $p\neq p'\in \N$.  Since $\delta(t^p)\neq \delta(t^{p'})$ are both prime, it follows that~$\Delta_{R^n}(t^p)$ and $\Delta_{R^n}(t^{p'})$ are coprime and so $\mathcal{Q}^p$ is strongly coprime to $\mathcal{Q}^{p'}$. Thus, assumption~\eqref{item:coprime} of Theorem~\ref{cor:lin-ind} is satisfied. 

Assumption~\eqref{item:rho} of Theorem~\ref{cor:lin-ind} follows immediately from the hypothesis when $n>1$.  When $n=1$, it follows from from the fact that $\mathcal{FOS}(R_{p,1}) = \mathcal{FOS}(R)$, for any $p\geq 1$ and any ribbon knot $R$, by equation~\eqref{eqn:rho and cable} from the proof of Theorem~\ref{prop:cablerobust}. The claimed result then follows from the observation that $J^m_{p,1}$ is isotopic to  $(R^{n}_{p,1})_{\alpha_{n}'}\circ R^{n-1}_{\alpha_{n-1}}\circ \cdots \circ R^{1}_{\alpha_{1}}(K^m)$.
\end{proof}

\noindent In the simplest case, where the sequence $\mathcal{R}$ is constant, we get the following corollary. 

\begin{corollary}~\label{cor:subgroupgen}
Let $R_\alpha$ be a robust doubling operator and $\Delta_{R}(t)$ be the Alexander polynomial of $R$.  
Let $\{K^{m}\}_{m \geq 1}$ be an infinite family of knots. 
Suppose that 
\begin{enumerate}
\item for all $p\in \N$, $\Delta_{R}(t^p) = \delta(t^p)\delta(t^{-p})$ where $\delta(t^p)$ is prime in $\Q[t,t^{-1}]$, and
\item no nontrivial  rational linear combination of elements of $\{\rho_0(K^{m})\}_{m\geq 1}$ produces an element in the rational span of $\mathcal{FOS}(R)$.
\end{enumerate}
Then $\{J^m_{p,1}\}_{m,p\geq 1}$ is linearly independent in $\C/(\mathcal{F}_{n.5}+\B_{n+1})$. where $J^m$ denotes the knot $\left(R_\alpha\right)^n(K^m)$.
\end{corollary}

Consider the twist knots $T_j$ with a positive clasp and $j$ full left-handed twists, as shown in Figure~\ref{fig:twist_knot}. Such knots have vanishing $\Arf$ invariant when the number of twists is even. In~\cite[Proposition~2.6, Section~5]{Cochran-Orr-Teichner:2004-1}, Cochran-Orr-Teichner showed that there exists an infinite collection of twist knots $\K = \{T_{2j_m}\}_{m\geq 1}$ with vanishing Arf invariant for which~$\{\rho_0(T_{2j_m})\}$ is linearly independent over $\Q$. 
\begin{figure}[h]
  \includegraphics[width=0.4\textwidth]{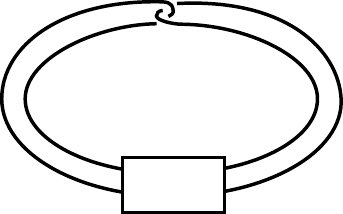}
  \put(-89,10){\LARGE $j$}
\caption{The twist knot $T_j$. The solid box containing $j$ denotes $j$ full \emph{left-handed} twists. Note that these result in positive crossings when $j$ is positive.}\label{fig:twist_knot}
\end{figure}

\begin{figure}[t]
  \includegraphics[width=0.9\textwidth]{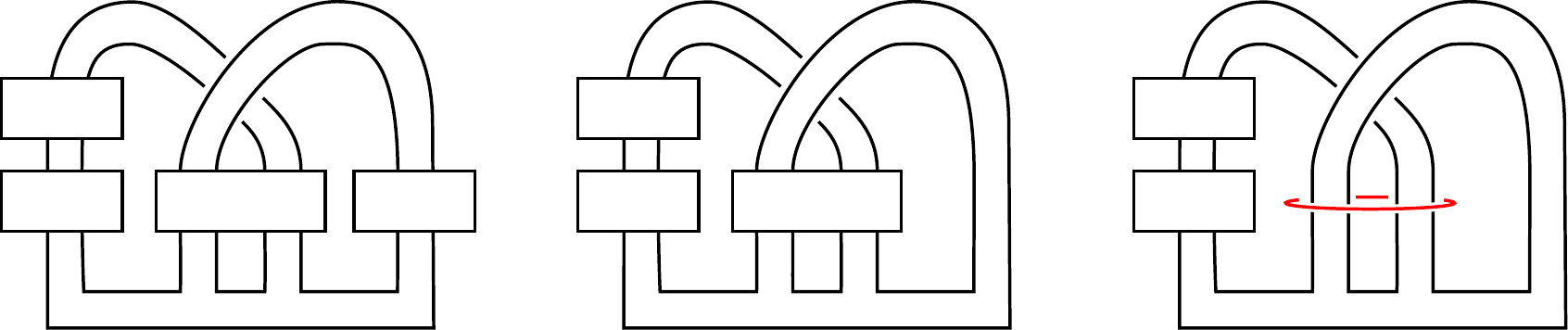}
  \put(-383,52){$T_{2j_m}$}
    \put(-380,28.3){$-k$}
        \put(-335,28.3){$J_k$}
                \put(-297,28.3){$k+1$}
    \put(-240,52){$T_{2j_m}$}
    \put(-237,28.3){$-k$}
        \put(-192,28.3){$J_k$}
    \put(-355,-16){$R^{k,J_k}_{\alpha_k}(T_{2j_m})$}
    \put(-190,-16){$R'$}
        \put(-102,52){$T_{2j_m}$}
    \put(-99,28.3){$-k$}
        \put(-27,30){$\eta$}
    \put(-55,-16){$U_\eta$}
\caption{Left: The result of infection $R^{k,J_k}_{\alpha_k}(T_{2j_m})$. Center: Changing~$k+1$ negative crossings gives a new knot $R'$. Right: A doubling operator~$U_\eta$ with unknotted pattern for which $U_\eta(J_k)=R'$.}\label{fig:bipolar}
\end{figure}
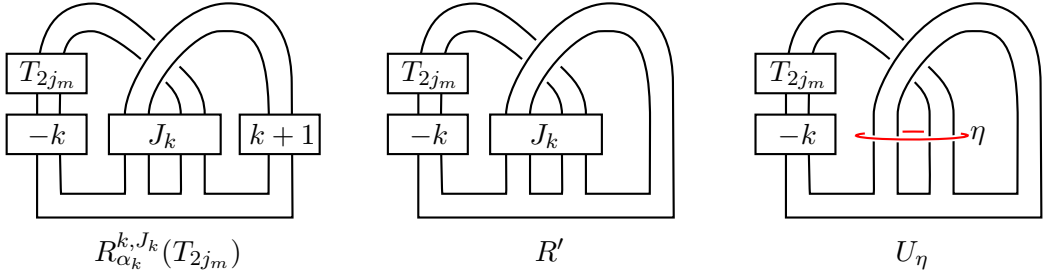

\noindent We recall Theorem~\ref{thm:cables-indep} and give the proof. 

\begin{thm:cables-indep}
For any $n\ge1$, there exists an infinite family of knots $\{K^i\}_{i\geq 1}\subset \F_n\cap \B_{n-1}$, such that the set of cables $\{K^i_{p,1}\}_{i,p\geq 1}$ is linearly independent in $\Fn/\Fnpointfive$ and in $\B_{n-1}/\B_{n+1}$.
\end{thm:cables-indep}

\begin{proof}

For any $k\in \N$, let $R^{k,J_k}_{\alpha_k}$ be the robust doubling operator of Proposition \ref{prop:robust-with-ambiguity} where $-\rho_0(J_k)\notin \mathcal{FOS}(R^{k,U})$ and $J_k\in \mathcal{N}_0$. This can be arranged, for example, by taking $J_k$ to be a connected sum of sufficiently many copies of the left-handed trefoil. Let $\{T_{2j_m}\}_{m\geq 1}$ be the collection of twist knots given above.  As $\{\rho_0( T_{2j_m})\}_{m\geq 1}$ is an infinite linearly independent set and $\mathcal{FOS}(R^{k,U})$ contains at most two non-zero elements, there exist some $m_1$, $m_2$ so that no nontrivial linear combination of $\{\rho_0(T_{2j_m})\mid m \neq m_1,m_2\}$ is in the rational span of~$\mathcal{FOS}(R^{k,J_k})$.  

We now show that the set $\{(R^{k,J_k}_{\alpha_k})^n(T_{2j_m})\mid m\neq m_1,m_2\}$ has the desired properties. First, since $\operatorname{Arf}(T_{2j_m})=0$, we know that $T_{2j_m}\in \F_0$. By Proposition~\ref{prop:infection}, the knot~$(R^{k,J_k}_{\alpha_k})^n(T_{2j_m})\in \F_n$. Next we show that $(R^{k,J_k}_{\alpha_k})^n(T_{2j_m})\in \B_{n-1}$. Since these twist knots can be unknotted by undoing a single positive crossing at the clasp, we see that~$T_{2j_m}\in \mathcal{P}_0$ by Proposition \ref{prop:postive}.  Thus, by Proposition~\ref{prop:infection}, the knot $R^{k,J_k}_{\alpha_k}(T_{2j_m})\in \mathcal{P}_1\subset \mathcal{P}_0$. As depicted in Figure~\ref{fig:bipolar}, changing $k+1$ negative crossings transforms $R^{k,J_k}_{\alpha_k}(T_{2j_m})$ to the satellite knot $U_\eta(J_k)$, where $U$ is the unknot.  By assumption,~$J_k\in \mathcal{N}_0$ and so by Proposition~\ref{prop:infection}, we see that $U_\eta(J_k)\in \mathcal{N}_1\subset \mathcal{N}_0$. Thus, $R^{k,J_k}_{\alpha_k}(T_{2j_m})\in \mathcal{N}_0$ by Proposition~\ref{prop:postive}.  We now conclude that $(R^{k,J_k}_{\alpha_k})(T_{2j_m})\in \mathcal{B}_{0}$ and so by Proposition~\ref{prop:infection}, we see that~$(R^{k,J_k}_{\alpha_k})^n(T_{2j_m})\in \mathcal{B}_{n-1}$. To summarize, we have now shown that~$(R^{k,J_k}_{\alpha_k})^n(T_{2j_m})\in \Fn\cap \mathcal{B}_{n-1}$.

It remains to verify the linear independence claim, for which we use Corollary~\ref{cor:subgroupgen}. Recall that $\Delta_{R^{k,J_k}}(t)=\delta_k(t)\delta_k(t^{-1})$, where $\delta_k(t)=kt-(k+1)$. As we checked in the proof of Corollary \ref{cor:robust2}, $kt^p-(k+1)$ is prime for every $p$. This immediately implies condition (1) of  Corollary~\ref{cor:subgroupgen}.  We have already explicitly arranged that condition (2) holds by restricting to $m\neq m_1,m_2$.  Thus, we have that 
$\{(R^{k,J_k}_{\alpha_k})^n(T_{2j_m})_{p,1} \mid  p\ge 1,  m\neq m_1,m_2\}
$
is a linearly independent set in $\F_{n}/\F_{n.5}$ and in $\B_{n-1}/\B_{n+1}$.  Letting $\{K_i\}$ be an enumeration of the countable set $\{(R^{k,J_k}_{\alpha_k})^n(T_{2j_m}) \mid m\neq m_1,m_2\}$ completes the proof.
 \end{proof}

\section{Proof of Theorem \ref{thm:summand}}\label{sec:proofofB}

We briefly review Legendrian knots; see ~\cite{Etnyre:2005-1} for further details. An embedding $\mathcal{LK}$ of a knot $K$ in $S^3$ is said to be \emph{Legendrian} if $\mathcal{LK}$ is tangent to the $2$-planes of the standard contact structure of $S^3$. Up to an isotopy we assume that $K$ misses the point at infinity and so lies in  $S^3-\{\infty\} = \mathbb R^3$.
 The \emph{front projection} of a Legendrian knot is obtained by then projecting to the $xz$-plane (e.g.\ the middle panel in Figure~\ref{fig:leg_doubling}). The classical invariants of Legendrian knots, namely the \emph{Thurston-Bennequin number}, denoted by $\tb(\cdot)$, and the \emph{rotation number}, denoted by $\rot(\cdot)$, may be computed from the front projection as follows. Let $\Pi(\mathcal{LK})$ be a front projection of a Legendrian knot $\mathcal{LK}$. Then 
\[
\tb(\mathcal{LK}) = \text{writhe $\left(\Pi(\mathcal{LK})\right) - \frac{1}{2} \#$ cusps  $\left(\Pi(\mathcal{LK})\right)$}
\]
and
\[
\rot(\mathcal{LK}) = \frac{1}{2} \#\text{ downward-moving cups$\left(\Pi(\mathcal{LK})\right) - \frac{1}{2} \#$ upward-moving cusps  $\left(\Pi(\mathcal{LK})\right)$}.
\]
Given any Legendrian knot $\mathcal{LK}$ representing a knot $K$, we can perform a \emph{positive stabilization} which preserves the topological knot type but produces a new Legendrian knot $\mathcal{LK}'$ such that $\tb(\mathcal{LK}')=\tb(\mathcal{LK})-1$ and $\rot(\mathcal{LK}')=\rot(\mathcal{LK})+1$. 

Let $\tau(\cdot)$ denote the concordance invariant of Ozsv\'ath-Szab\'o \cite{Ozsvath-Szabo:2004-0} from Heegaard-Floer homology. The following inequality was proved by Plamenevskaya. 

\begin{theorem}[{\cite[Theorem~1]{Plamenevskaya:2004-1}}]~\label{thm:tbineq} Let $\mathcal{LK}$ be a Legendrian embedding of a knot $K$ in~$S^3$. Then 
$$\tb(\mathcal{LK}) + |\rot(\mathcal{LK})| \leq 2\tau (K) -1.$$
\end{theorem}

The pattern knot $R$ corresponding to any doubling operator $R_\alpha$ can be considered as a knot within $S^1 \times \mathbb{R}^2=\Int(S^3\setminus N(\alpha))$. A doubling operator $R_\alpha$ is said to be a \emph{Legendrian doubling operator} if the embedding $\mathcal{LR}_\alpha$ of $R$ in $S^1 \times \mathbb{R}^2$ is tangent to the $2$-planes of the standard contact structure on $S^1 \times \mathbb{R}^2$. We also obtain front projections as before on the~$xz$-plane, where the $x$-direction is understood to be periodic. We can also draw a single periodic domain as in the first panel in Figure~\ref{fig:leg_doubling}. Using front projections, we can define the Thurston-Bennequin number and the rotation number of a Legendrian doubling operator using the same combinatorical formulae as for Legendrian knots given above.

\begin{figure}[h]
  \includegraphics[width=\textwidth]{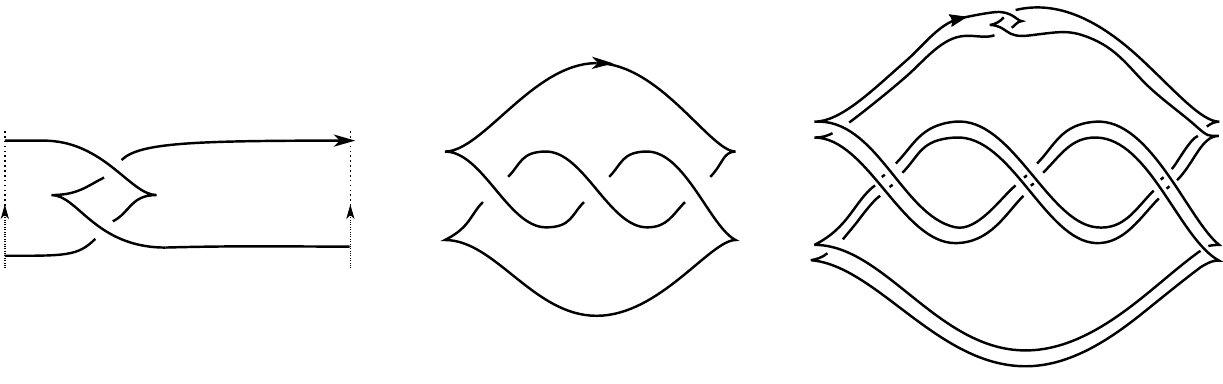}
  \put(-375,-20){$\mathcal{LR}_\alpha$}
  \put(-230,-20){$\mathcal{LK}$}
  \put(-85,-20){$\mathcal{LR}_\alpha(\mathcal{LK})$}
  \put(-395,-35){$\tb(\mathcal{LR}_\alpha)=1$}
  \put(-398,-50){$\rot(\mathcal{LR}_\alpha)=0$}
  \put(-250,-35){$\tb(\mathcal{LK})=1$}
  \put(-253,-50){$\rot(\mathcal{LK})=0$}
  \put(-105,-35){$\tb(\mathcal{LR}_\alpha(\mathcal{LK}))=1$}
  \put(-108,-50){$\rot(\mathcal{LR}_\alpha(\mathcal{LK}))=0$}
\caption{The Legendrian doubling operation.}\label{fig:leg_doubling}
\end{figure}
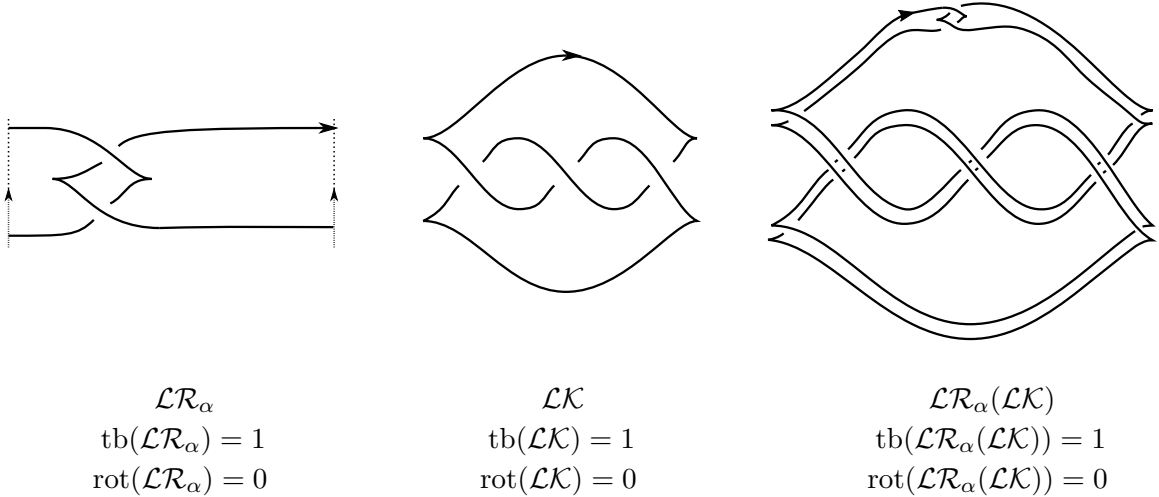

Let $\mathcal{LR}_\alpha$ be a Legendrian doubling operator in $S^1 \times \mathbb{R}^2$ with $2n$ end points, and $\mathcal{LK}$ be a Legendrian knot in $S^3$. Let $\mathcal{LR}_\alpha(\mathcal{LK})$ be the Legendrian knot obtained by taking $n$ vertical parallel copies of $\mathcal{LK}$ and inserting $\mathcal{LR}_\alpha$ in an appropriately oriented strand of $\mathcal{LK}$ (see the third panel in Figure~\ref{fig:leg_doubling} for an example). This is called the \emph{Legendrian doubling operation} or the \emph{Legendrian satellite operation}; see also \cite{Ng:2001-1,Ng-Traynor:2004-1,Ray:2015-1,Park-Ray:2018-1}. Observe that when $\tb(\mathcal{LK})=0$, $\mathcal{LR}_\alpha(\mathcal{LK})$ is a Legendrian representative of $R_\alpha(K)$. The following proposition is useful to compute the Thurston-Bennequin number and the rotation number of $\mathcal{LR}_\alpha(\mathcal{LK})$ (see also \cite[Remark~2.4]{Ng:2001-1}).

\begin{proposition}~\label{prop:tbdoubling} For a Legendrian doubling operator $\mathcal{LR}_\alpha$ and a Legendrian knot $\mathcal{LK}$, $$\tb \left( \mathcal{LR}_\alpha(\mathcal{LK})\right) =\tb\left( \mathcal{LR}_\alpha\right) \text{ and }\rot\left( \mathcal{LR}_\alpha(\mathcal{LK})\right) =\rot\left( \mathcal{LR}_\alpha\right).$$
\end{proposition}

In what follows we shall restrict ourselves to the doubling operators $Q_{\alpha_k}^{k}$, which we showed to be robust in Proposition~\ref{prop:robust} when $k\ge 3$.  

\begin{proposition}~\label{prop:taucomp} Let $K$ be a knot with a Legendrian representative with non-negative Thurston-Bennequin number. Then $\tau\left(\left(Q_{\alpha_k}^{k}\right)^n(K)\right) = 1$ for all $k, n\geq 1$. 
\end{proposition}
Notice that any twist knot with $j$ full left-handed twists has a Legendrian representative with non-negative Thurston-Bennequin number. This is given in Figure~\ref{fig:legtwistknots}.
\begin{figure}[htb]
\labellist
\small\hair 2pt
\pinlabel $j-1$ at 22 35
\endlabellist
\centering
\includegraphics[height=3.5cm]{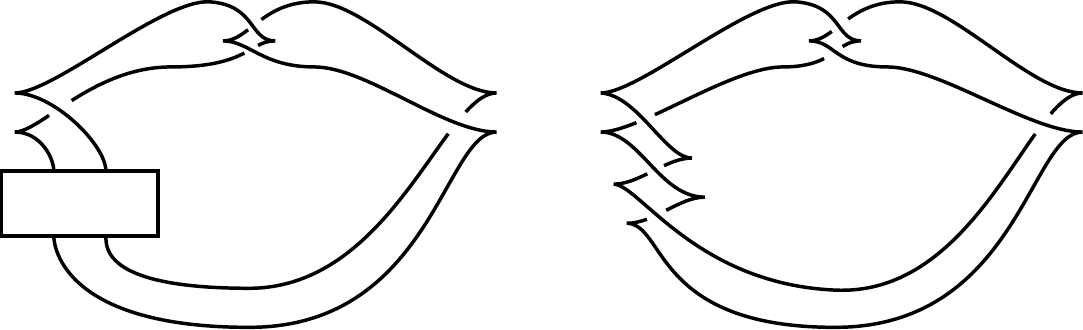}
\caption{Left: A Legendrian diagram $\mathcal{LT}_j$ for the twist knot $T_{j}$ with $j$ full negative twists. The box containing $j-1$ denotes $j-1$ full left-handed twists. We compute that $\tb\left(\mathcal{LT}_j\right)=1$ and  $\rot\left(\mathcal{LT}_j\right)=0$ for all $j\geq 1$. Right: The specific Legendrian diagram $\mathcal{LT}_3$ for $j=2$.}\label{fig:legtwistknots}
\end{figure}

\begin{proof} Observe that $\mathcal{LQ}_{\alpha_k}^{k}$, shown in Figure~\ref{fig:robustleg}, 
\begin{figure}[htb]
\labellist
\small\hair 2pt
\pinlabel $k$ at 40 150
\pinlabel $\alpha_k$ at 25 117
\pinlabel $(a)$ at 48 75

\pinlabel $k$ at 182 150
\pinlabel $\alpha_k$ at 167 117
\pinlabel $(b)$ at 185 75

\pinlabel $k$ at 332 150
\pinlabel $(c)$ at 330 75

\pinlabel $(d)$ at 185 0

\endlabellist
\centering
\includegraphics{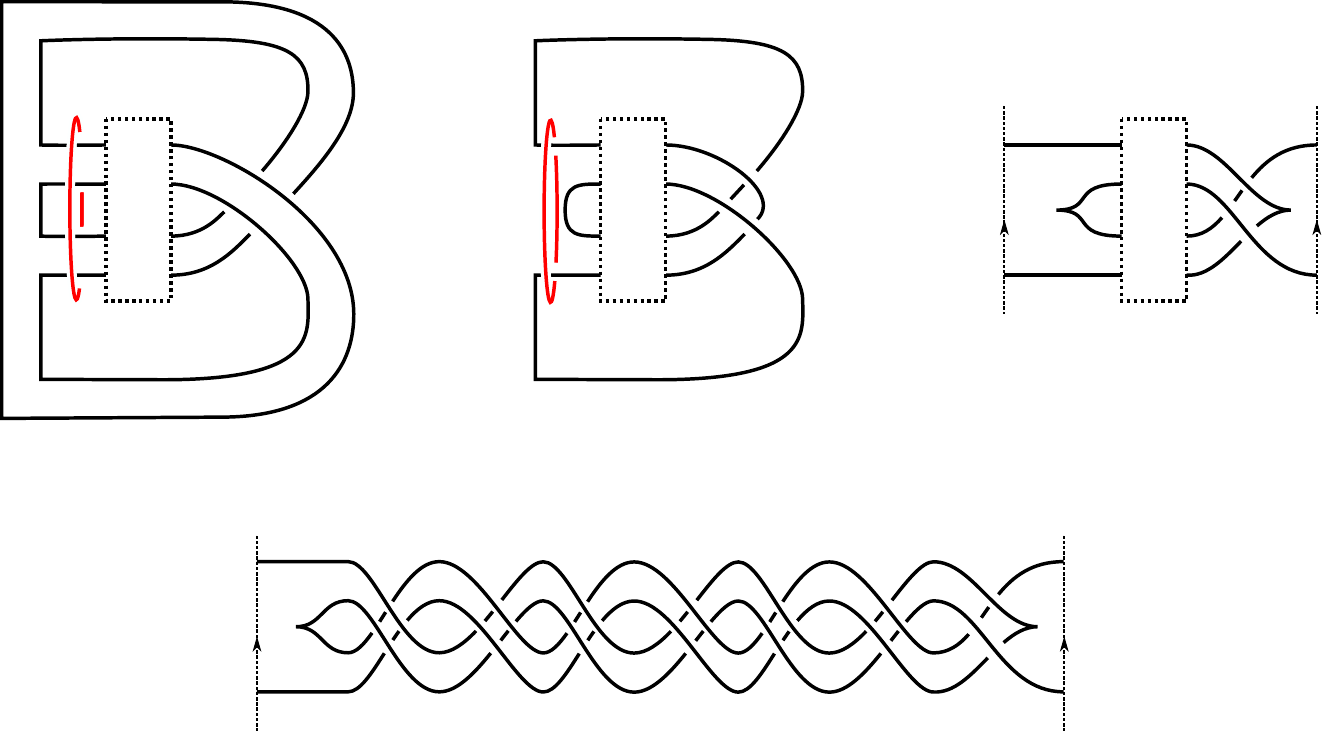}
\caption{(a) The robust doubling operator $Q^k_{\alpha_k}$. (b) The result of an isotopy. (c) A Legendrian diagram $\mathcal{LQ}_{\alpha_k}^{k}$ for the robust doubling operator $Q_{\alpha_k}^{k}$. We compute that $\tb(\mathcal{LQ}_{\alpha_k}^{k})=0$ and  $\rot(\mathcal{LQ}_{\alpha_k}^{k})=1$. (d) The specific Legendrian diagram $\mathcal{LQ}_{\alpha_3}^{3}$ for $Q^3_{\alpha_3}$.}\label{fig:robustleg}
\end{figure}
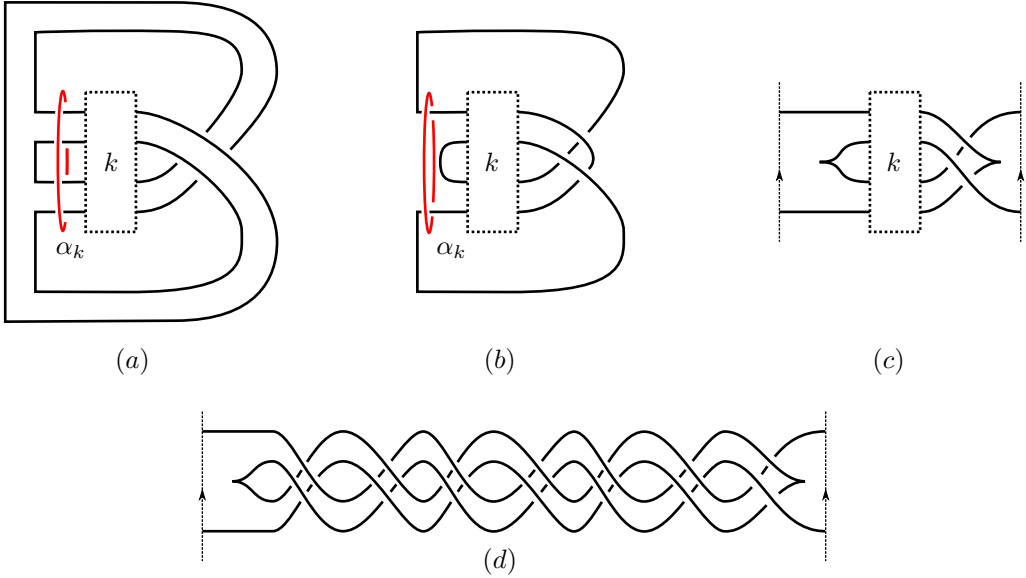
is a Legendrian representative of $Q^{k}$. Given a Legendrian representative  for $K$ with non-negative Thurston-Bennequin number, stabilize to obtain a Legendrian representative $\mathcal{LK}$ such that $\tb(\mathcal{LK})=0$. Then apply the formulae from Proposition~\ref{prop:tbdoubling} to conclude that $$\tb \left( \mathcal{LQ}_{\alpha_k}^{k}(\mathcal{LK})\right) =\tb\left(\mathcal{LQ}_{\alpha_k}^{k}\right) =0 \text{ and } \rot \left( \mathcal{LQ}_{\alpha_k}^{k}(\mathcal{LK})\right) =\rot \left(\mathcal{LQ}_{\alpha_k}^{k}\right) =1.$$ Recall that the $\tau$ invariant gives a lower bound for the Seifert genus of a knot \cite{Ozsvath-Szabo:2003-1}. Since~ $Q_{\alpha_k}^{k}(K)$ has genus one, we conclude that $\tau\left(Q_{\alpha_k}^{k}(K)\right) = 1$ by Theorem~\ref{thm:tbineq}. 

Note that $Q_{\alpha_k}^{k}(K)$ has a Legendrian representative $\mathcal{LQ}_{\alpha_k}^{k}(\mathcal{LK})$ with zero Thurston-Bennequin number. This allows us to repeat this process and conclude that 
\[
\tau\left(\left(Q_{\alpha_k}^{k}\right)^n(K)\right) = 1 \text{ for all } k, n\geq 1.\qedhere
\] 
\end{proof}

The following theorem of Feller and the second and third authors is the last ingredient needed for our proof of Theorem~\ref{thm:summand}. 

\begin{theorem}[{\cite[Theorem~5.21]{Feller-Park-Ray:2016-1}}]~\label{thm:Upsilon}For any genus one knot $K$ with $\tau(K)=1$, the knots $\{K_{2^j,1}\}_{j \geq 0}$ spans an infinite rank summand of $\mathcal{C}$.
\end{theorem}

\noindent We recall Theorem~\ref{thm:summand} and give the proof.

\begin{thm:summand} For any $n\geq 0$, there exists an infinite family of knots $\{K^i\}_{i\ge 1}\subset  \F_n$ such that for each fixed $i$ the set $\{K^i_{2^j,1}\}_{j\geq 0}$ is a basis for an infinite rank summand of $\Fn$ and for which $\{K^i_{p,1}\}_{i, p\geq 1}$ is linearly independent in $\mathcal{F}_{n}/\mathcal{F}_{n.5}$.\end{thm:summand}

\begin{proof} For $n=0$, this follows from classical results~\cite{Levine:1969-1, Litherland:1979-1} along with the facts that a knot $K$ is $0$-solvable if and only if $\Arf(K)=0$ and that a knot $K$ is $0.5$-solvable if and only if $K$ is algebraically slice~\cite[Remark 1.3.2, Theorem 1.1]{Cochran-Orr-Teichner:2003-1}. 
Now, fix a positive integer $n$ and consider the robust doubling operators $Q^k_{\alpha_k}$ for $k\ge 3$.  Let $\{T_{2j_m}\}$ be the set of twist knots for which~$\{\rho_0(T_{2j_m})\}$ is linearly independent over $\Q$ from the proof of Theorem~\ref{thm:cables-indep}.  Notice that $\mathcal{FOS}(Q^k)$ has exactly one nonzero element, so that just as in the proof of Theorem \ref{thm:cables-indep}, there is some $m_1\in \N$ so that $\{(Q^k_{\alpha_k})^n(T_{2j_m})_{p,1}\mid m\neq m_1, p\ge 1\}$ is a linearly independent subset of $\F_n/\F_{n.5}$.  

Now fix $k\ge 3$ and $m\neq m_1$.  We now show that $\{(Q^k_{\alpha_k})^n(T_{2j_m})_{2^j,1}\}_{j\geq 0}$ generates an infinite rank summand of $\F_n$.  By Proposition~\ref{prop:taucomp}, $\tau((Q^k_{\alpha_k})^n(T_{2j_m})) = 1$ and by construction~$(Q^k_{\alpha_k})^n(T_{2j_m})$ has genus one. By Theorem~\ref{thm:Upsilon}, we conclude that $\{(Q^k_{\alpha_k})^n(T_{2j_m})_{2^j,1}\}_{j\geq 0}$ spans an infinite rank summand of $\C$ and therefore, of~$\mathcal{F}_{n}$. As in the proof of Theorem \ref{thm:cables-indep}, letting $\{K_i\}$ be an enumeration of $\{(Q^k_{\alpha_k})^n(T_{2j_m})\mid m\neq m_1\}$ completes the proof.
\end{proof}

\begin{remark}
If one is not interested in cables, one could construct an infinite rank summand of any $\mathcal{F}_{n}$ whose image in $\mathcal{F}_{n}/\mathcal{F}_{n.5}$ is any fixed infinite rank subgroup via an abstract linear algebra argument using the fact that $\C$ has an infinite rank summand contained in~$\bigcap \Fn$. More explicitly, we can do the following construction. Fix $n\geq 0$. Let~$\{J^m\}_{m\geq 1}$ be a basis for an infinite rank summand of the subgroup of $\C$ consisting of smooth concordance classes of topologically slice knots, produced by~\cite{Ozsvath-Stipsicz-Szabo:2017-1} (the first such summand was found by Hom in~\cite{Hom:2015-3} using her $\varepsilon$-invariant). In particular, the sequence of first singularities of $\{\Upsilon(J^m)\}_{m\geq 1}$ is monotone decreasing and converges to $0$, and the slope change at the first singularities is the least allowed. Now, let $K$ be any nontrivial knot in $\Fn/\Fnpointfive$ with a genus one Seifert surface. Then the first singularity of $\Upsilon(J^m\# K)$, as well as the slope change there,  coincides with those of $\Upsilon(J^m)$. This follows since $\Upsilon(K)$ is either the zero function or has a unique singularity at $1$ (see~\cite[Proposition~3.1]{Feller-Park-Ray:2016-1}). Let~$\{K^m\}_{m\geq 1}$ be a set of genus one knots which is linearly independent in $\Fn/\Fnpointfive$, produced by~\cite{Cochran-Harvey-Leidy:2009-1}. Then the sequence of first singularities of $\{\Upsilon(J^m\#K^m)\}_{m\geq 1}$ is monotone decreasing and converges to~$0$, and the slope change at each first singularity is the least allowed. Thus,~$\{J^m\# K^m\}_{m\geq 1}$ is a basis for an infinite rank summand of $\Fn$ and projects to the linearly independent set~$\{K^m\}_{m\geq 1}$ in $\Fn/\Fnpointfive$. Similar constructions suffice to show the existence of a basis for an infinite rank summand of $\mathcal{F}_n$ projecting to any choice of infinite rank subgroup of~$\Fn/\Fnpointfive$. In that case, one should use a subsequence of~$\{J^m\}_{m\geq 1}$ to ensure the condition on first singularities.
\end{remark}

\section{Proof of Theorems~\ref{thm:kauffman}~and~\ref{thm:kauffmanprime}}\label{sec:proofofC}

\noindent Recall the following theorem of Cochran and the first author. 

\begin{theorem}[{\cite[Theorem~1.5 and Proposition~5.2]{Cochran-Davis:2015-1}}]~\label{thm:kauffmanD} There exists a slice knot $K$ with a genus one Seifert surface and derivative curves $d$ and $d'$, such that $d$ and $d'$ have non-vanishing $\Arf$ invariant and Levine-Tristram signature function. 

More precisely for any knot $L$, there exists a slice knot $K_L$ with the genus one Seifert surface and derivative curves $d = L \# - (L_{2,1})$ and $d'$ described in Figure~\ref{fig:CDDeriv}.  \end{theorem}

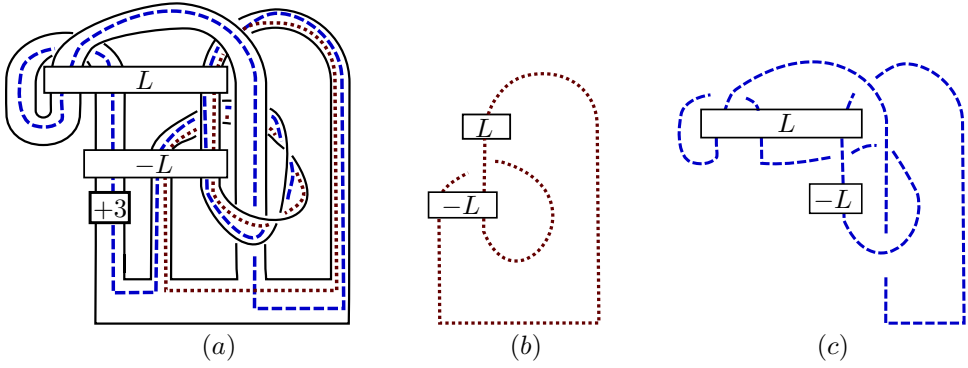
\begin{figure}[htb]
\labellist
\small\hair 2pt
\pinlabel $-L$ at 160 173
\pinlabel $L$ at 150 265
\pinlabel $+3$ at 113 125
\pinlabel $(a)$ at 230 -25
\pinlabel $-L$ at 490 128
\pinlabel $L$ at 515 213
\pinlabel $(b)$ at 560 -25
\pinlabel $-L$ at 890 135
\pinlabel $L$ at 840 218
\pinlabel $(c)$ at 890 -25
\endlabellist
\centering
\includegraphics[width=.3\textwidth]{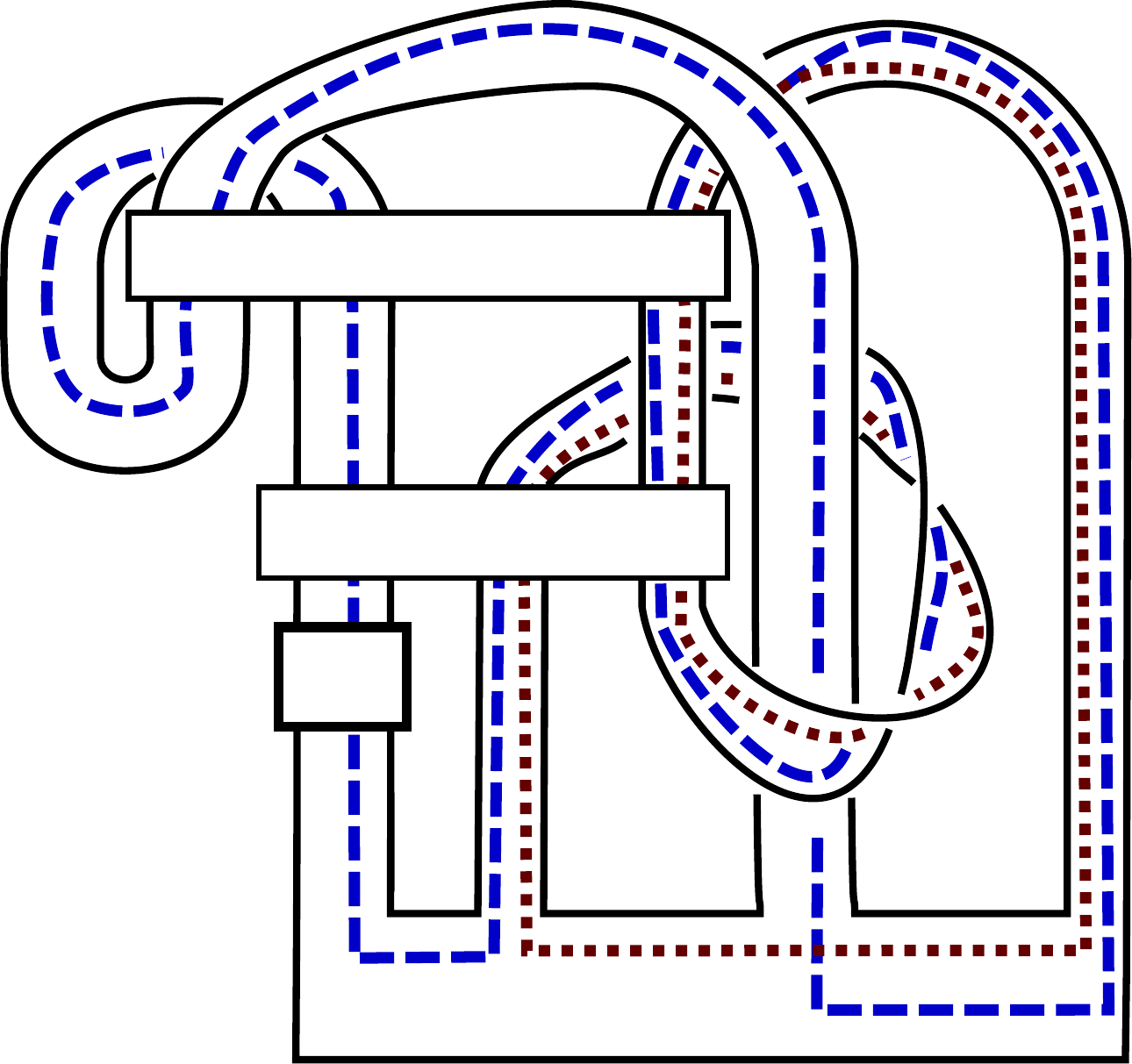}
\hspace{.05\textwidth}
\includegraphics[width=.15\textwidth]{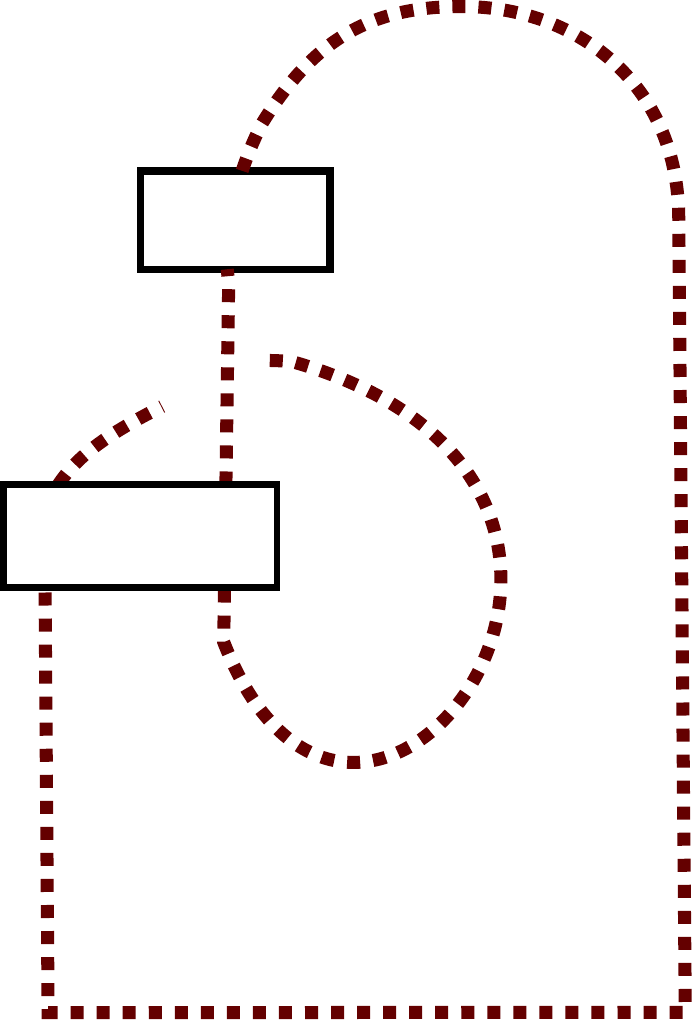}
\hspace{.05\textwidth}
\includegraphics[width=.25\textwidth]{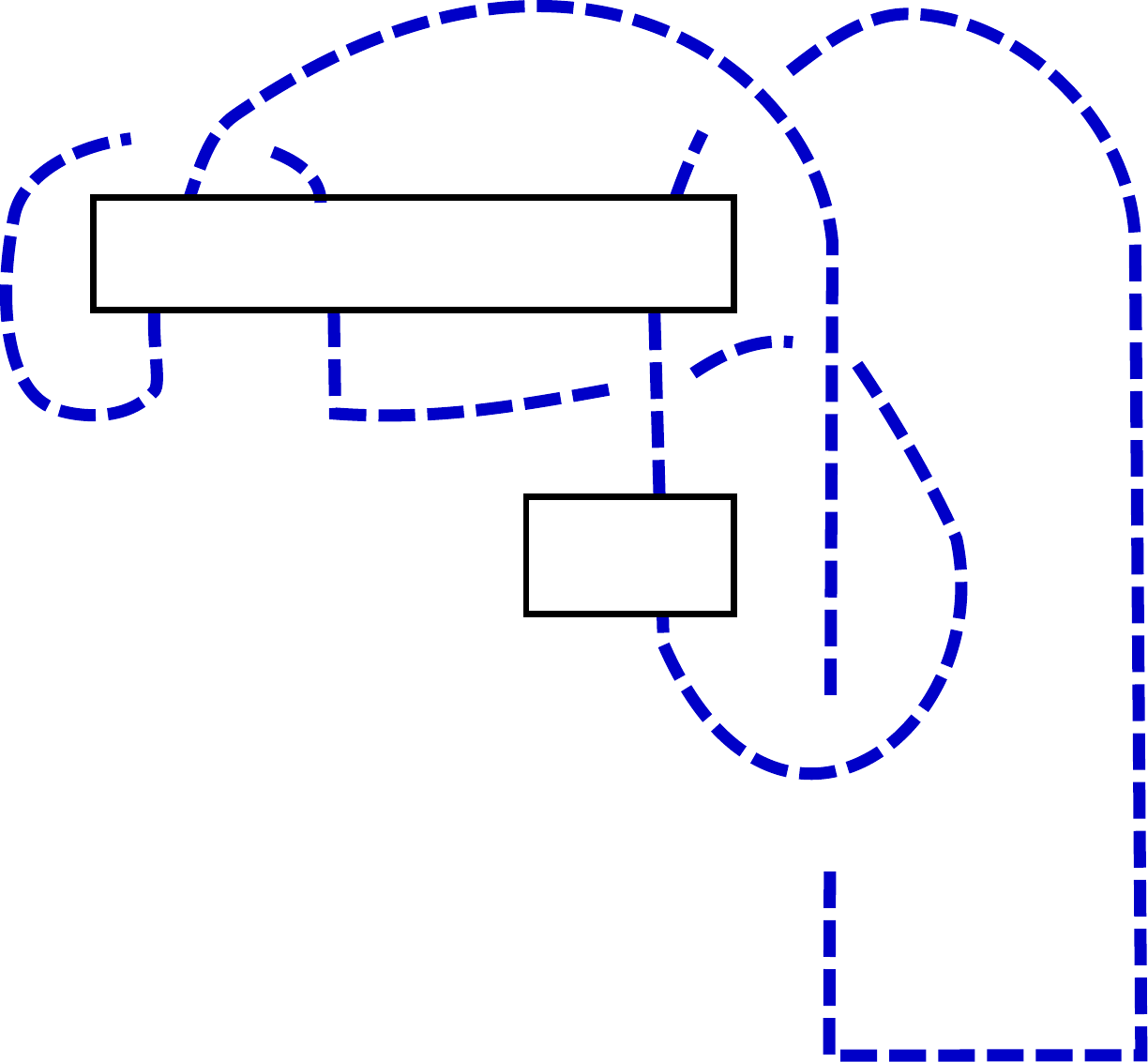}
\hspace{.05\textwidth}

\vspace{5mm}
\caption{In \cite[Proposition 5.1]{Cochran-Davis:2015-1}, the knot $K_L$, shown in (a), is proved to be slice for any choice of knot $L$. On the obvious Seifert surface one sees the derivatives $d=L\#-L_{2,1}$, shown in (b), and $d'$, shown in (c), with the same Arf invariant and Levine-Tristram signature as the left-handed trefoil knot.}\label{fig:CDDeriv}
\end{figure}
\noindent Next we recall and prove Theorems~\ref{thm:kauffman} and~\ref{thm:kauffmanprime}.

\begin{thm:kauffman} For any $n\geq 0$, there exists a slice knot $K$ bounding a genus one Seifert surface with derivative curves $d$ and $d'$ such that $\Arf(d')\neq 0$ and $d$ is nontrivial in each of the quotients $\Fn / \Fnpointfive$ and $\B_{n-1} / \B_{n+1}$.
\end{thm:kauffman}

\begin{proof} 
We will choose slice knots as in Theorem~\ref{thm:kauffmanD} and vary the choice of $L$. We first show that $d'$ has non-vanishing $\Arf$ invariant for any choice of $L$. Recall that $\Arf$ is a $\mathbb{Z}_2$-valued invariant and the effect of the satellite construction and connected sum on the $\Arf$ invariant is well understood (see for instance~ \cite[Corollary~2.3]{Cochran-Davis:2015-1}). Note that $d'$ is obtained from the right-handed trefoil $T$ by a pair of winding number $\pm 1$ satellite constructions with companions $L$ and $-L$. Thus, we have the following computation:
\begin{align*}
  \Arf(d') &= \Arf(T) + \Arf(L) + \Arf(-L)  \\
&= \Arf(T) = 1.
\end{align*}

Next, we need to show that $d$ is nontrivial in $\Fn/\Fnpointfive$ and $\mathcal{B}_{n-1}/\mathcal{B}_{n+1}$, for some choice of $L$. For $n=0$, the statement was proved by Cochran and Davis in~\cite[Section~5]{Cochran-Davis:2015-1} (see also \cite{Park:2018-1}) by choosing $L$ to be a knot with 
non-vanishing $\sigma_{\omega^2}(L) - \sigma_{\omega}(L)$ (e.g.\ the connected sum of two right-handed trefoils). Let~$n\geq 1$ and choose $L$ to be of the form~$K$ as in Theorem \ref{thm:cables-indep}. Since $\{K_{p,1}\}_{p\geq 1}$ is linearly independent in $\mathcal{F}_n/\mathcal{F}_{n.5}$ and $\mathcal{B}_{n-1}/\mathcal{B}_{n+1}$, we see that, in particular, $d =  K\# - K_{2,1}$ is nontrivial in $\mathcal{F}_n/ \mathcal{F}_{n.5}$ and in $\mathcal{B}_{n-1}/\mathcal{B}_{n+1}$, as needed. \end{proof}

\begin{thm:kauffmanprime}
There exists a slice knot $K$ with a genus one Seifert surface and derivative curves $d$ and $d'$, such that $d'$ has non-vanishing $\Arf$ invariant and $d$ is topologically slice but not $($smoothly$)$ slice.
\end{thm:kauffmanprime}

\begin{proof} Let $K$ be a slice knot from Theorem~\ref{thm:kauffmanD} where we choose $L$ to be the positive Whitehead double of the right-handed trefoil, denoted $D$. From the proof of Theorem~\ref{thm:kauffman}, we know that $d'$ has non-vanishing $\Arf$ invariant. The knot $d$ is not (smoothly) slice; this can be seen, for example, using the~$\tau$-invariant, as follows. Since $\tau(D)=1$~\cite{Livingston:2004-1}  (see also~\cite{Hedden:2007-1}), we compute that $\tau(D_{2,1})=2$ using Hom's cabling formula from~\cite[Theorem 1]{Hom:2014-2}. Since the $\tau$-invariant is additive under connected sum, $\tau(d) = \tau(D \# - D_{2,1}) = -1$. Lastly,~$d$ is topologically slice since it has trivial Alexander polynomial~\cite{Freedman:1982-2,Freedman-Quinn:1990-1,Garoufalidis-Teichner:2004-1}. \end{proof}

We end this section with an interesting corollary of Theorems~\ref{thm:kauffman} and~\ref{thm:kauffmanprime}. In a forthcoming paper we describe the knots for which condition (2) holds~\cite{Davis-Park-Ray:2018-1}.

\begin{corollary}\label{cor:handlebody}~
\begin{enumerate}
\item There exists a slice knot $K$ with a genus one Seifert surface $F$ such that there exists a topological slice disk $\Delta_{\text{top}}$ in $B^4$ for $K$ where $F\cup \Delta_{\text{top}}$ bounds a topological handlebody in $B^4$ but there is no $($smooth$)$ slice disk $\Delta$ in $B^4$ for $K$ where $F\cup \Delta$ bounds a smooth handlebody in $B^4$. 
\item For any $n\geq 0$, there exists a slice knot $K$ with a genus one Seifert surface $F$ such that there exists a smoothly embedded disk $\Delta_{n}$ in an $(n+1)$-solution $W$ for $K$ where $F\cup \Delta$ bounds a smooth handlebody in $W$ but there is no topological slice disk $\Delta$ in $B^4$ for $K$ where $F\cup \Delta$ bounds a topologically embedded handlebody in $B^4$.
\end{enumerate}
\end{corollary}
We remark that results similar to $(2)$ hold for $n$-positons and $n$-negatons via an identical proof.  

\begin{proof} As a preliminary observation, observe that if $K$ is a slice knot with a genus one Seifert surface $F$ and a slice disk $\Delta$ such that $F\cup \Delta$ bounds a smooth handlebody in $B^4$, then there exists a derivative curve on $F$ which is itself a slice knot.  Indeed, a meridional disk for that handlebody is bounded by such a derivative. Thus, if $K$ has no such derivative, there is no slice disk $\Delta$ such that $F\cup \Delta$ bounds a smooth handlebody in $B^4$. By the same argument, if $K$ has no topologically slice derivative, there is no topological slice disk $\Delta_{\text{top}}$ for $K$ such that $F\cup \Delta_{\text{top}}$ bounds a topological handlebody in $B^4$.

\textit{Proof of $(1)$:} Choose a slice knot $K$ from Theorem~\ref{thm:kauffmanprime}; in particular $K$ has a genus one Seifert surface $F$ with derivative curves $d$ and $d'$ where $\Arf(d')\neq 0$ and $d$ is topologically slice but not smoothly slice.  Consider a decomposition of $B^4 = (S^3 \times I) \cup (B^4)'$ and a smooth embedding $i \colon F \times I \hookrightarrow S^3 \times I \subset B^4$. The image $S:=i(F\times I)$ is an embedded genus two handlebody with boundary consisting of the union of the Seifert surface $F=i(F\times \{0\})$, an annulus $i(K\times I)$, and a pushed in Seifert surface $-F=i(F\times \{1\})$. Since the derivative $d$ is topologically slice, there exists a topological slice disk $\Delta_{d}$ in $(B^4)'$. Using the fact that $\Delta_{d}$ is locally flat, we may find a topological normal bundle. The trivialization of this normal bundle restricts to give the $0$-framing on the knot $d$ in $S^3\times \{1\}$. Since $d$ is a derivative, the surface framing induced by $F$ is the same as the $0$-framing on $d$. Thus, an annular neighborhood of $d$ in $F\times \{1\}$ extends to a copy of $\Delta_{d} \times I$ in $B^4’$. 

Glue $S$ and $\Delta_{d} \times I$ along this annular neighborhood of $d$ to produce a genus one topological handlebody $S'$; observe that $\partial{S'} = F \cup A \cup \Delta'_{\text{top}}$ where $\Delta'_{\text{top}}$ is a locally flat embedded disk in $(B^4)'$. By taking $\Delta_{\text{top}} = A \cup \Delta'_{\text{top}}$ we get the desired result.

\textit{Proof of $(2)$:} Choose a slice knot $K$ from Theorem~\ref{thm:kauffman}; that is, $K$ has a genus one Seifert surface $F$ and derivative curves $d$ and $d'$ such that $\Arf(d')\neq 0$ and $d$ is nontrivial in~$\F_n / \F_{n.5}$. Let $W'$ be an $n$-solution for $d$ and let $\Delta_{d}$ be a smoothly embedded disk in~$W'$ that satisfies the conditions from Definition~\ref{defn:n-solvable}. Consider $W := (S^3 \times I) \cup W'$ and a smooth embedding~$i \colon F \times I \hookrightarrow S^3 \times I \subset W$. We will show that $W$ is an $(n+1)$-solution for $K$. First construct a smoothly embedded disk $\Delta'$ in~$W'$ as follows. Let $F'$ be the image $i(F\times \{1\})$ embedded in the boundary of $W'$. Since $d$ bounds the smoothly embedded disk $\Delta_{d}$, we can cut open $F'$ along $d$ and glue on two copies of $\Delta_{d}$ to get a disk $\Delta'$ in $W'$. Then we can find a smoothly embedded disk, denoted by $\Delta$, for $K$ in $W$ by gluing together the image $i(K\times I)$ and $\Delta'$.  Similarly to the proof of $(1)$ we see that $\Delta \cup F$ bounds the handlebody obtained by gluing together $i(F\times[0,1])$ and $\Delta_d\times[0,1]$.  We still must show that $\Delta\subset W$ satisfies the definition of an $n$-solution (Definition \ref{defn:n-solvable}).  

 By thickening $\Delta_{d}$ we see that there is a smooth embedding
\[
\left(W' \setminus \Delta_{d} \times I\right) \subset \left(W' \setminus \Delta_{d} \times \{0\} \sqcup \Delta_{d} \times \{1\}\right) \subset \left(W \setminus \Delta\right).
\]
We know that $H_2(W) = H_2(W')$ has a basis consisting of $2k$ embedded, connected, compact, oriented surfaces $L_1,\dots ,L_k,D_1,\dots,D_k$ in the exterior of the disk $\Delta_{d}$; hence these lie in the exterior of $\Delta$. These surfaces have trivial normal bundles and satisfy the conditions of Definition~\ref{defn:n-solvable}. It only remains to show that $\pi_1(L_i) \subset \pi_1(W\setminus \Delta)^{(n+1)}$ and~$\pi_1(D_i) \subset \pi_1(W\setminus \Delta)^{(n+1)}$ for all $i=1,\dots,k$.  Observe that $H_1(W'-\Delta_d)\cong \Z$ is generated by the meridian of $d$. As the meridian of $d$ is nullhomologous in the exterior of $K$, we see that~$H_1(W'-\Delta_d)\to H_1(W-\Delta)$ is the zero homomorphism, and~$\pi_1(W'-\Delta_d)\subset \pi_1(W-\Delta)^{(1)}$. The functoriality of the derived series implies that 
$$\pi_1(L_i) \subset \pi_1(W'\setminus \Delta_d)^{(n)} \subset \left(\pi_1(W\setminus \Delta)^{(1)}\right)^{(n)}\subset \pi_1(W\setminus \Delta)^{(n+1)}.$$  Similarly, $\pi_1(D_i) \subset \pi_1(W\setminus \Delta)^{(n+1)}$, completing the proof.
\end{proof}

\bibliographystyle{amsalpha}
\renewcommand{\MR}[1]{}
\bibliography{research}
\end{document}